\documentclass[12pt]{amsart}
\usepackage{amsmath,amsthm,amsfonts,amssymb,mathrsfs}
\usepackage[all]{xy}
\usepackage{amssymb}
\usepackage{latexsym}
\usepackage{amsmath,amsthm}
\usepackage{aeguill}
\usepackage{amssymb}
\usepackage{mathrsfs}
\usepackage{hyperref}
\usepackage{etoolbox}
\usepackage{enumerate}

\makeatletter
\pretocmd{\chapter}{\addtocontents{toc}{\protect\addvspace{15\p@}}}{}{}
\pretocmd{\section}{\addtocontents{toc}{\protect\addvspace{3\p@}}}{}{}
\makeatother

\makeatletter
\def\@tocline#1#2#3#4#5#6#7{\relax
  \ifnum #1>\c@tocdepth 
  \else
    \par \addpenalty\@secpenalty\addvspace{#2}%
    \begingroup \hyphenpenalty\@M
    \@ifempty{#4}{%
      \@tempdima\csname r@tocindent\number#1\endcsname\relax
    }{%
      \@tempdima#4\relax
    }%
    \parindent\z@ \leftskip#3\relax \advance\leftskip\@tempdima\relax
    \rightskip\@pnumwidth plus4em \parfillskip-\@pnumwidth
    #5\leavevmode\hskip-\@tempdima
      \ifcase #1
       \or\or \hskip .5em \or \hskip 1em \else \hskip 1.5em \fi%
      #6\nobreak\relax
    \dotfill\hbox to\@pnumwidth{\@tocpagenum{#7}}\par
    \nobreak
    \endgroup
  \fi}
\makeatother
\newcommand{\C}{\mathbb{C}}
\newcommand{\N}{\mathbb{N}}
\newcommand{\Z}{\mathbb{Z}}

\newcommand{\Q}{\mathbb{Q}}
\newcommand{\F}{\mathbb{F}}

\newcommand{\G}{\mathbb{G}}

\newcommand{\cG}{\mathcal{G}}
\newcommand{\cH}{\mathcal{H}}
\newcommand{\cI}{\mathcal{I}}
\newcommand{\cS}{\mathcal{S}}
\newcommand{\cT}{\mathcal{T}}

\newcommand{\Gal}{\operatorname{Gal}}

\newcommand{\diag}{\operatorname{diag}}

\newcommand{\Spec}{\operatorname{Spec}\,}
\newcommand{\Res}{\operatorname{Res}}

\newcommand{\rk}{\operatorname{rk}}

\newcommand{\im}{\operatorname{im}}

\renewcommand{\sc}{\operatorname{sc}}
\renewcommand{\ss}{\operatorname{ss}}
\newcommand{\der}{\operatorname{der}}
\newcommand{\red}{\operatorname{red}}
\newcommand{\rank}{\operatorname{rank}}

\newcommand{\Hom}{\operatorname{Hom}}
\newcommand{\Aut}{\operatorname{Aut}}
\newcommand{\End}{\operatorname{End}}

\newcommand{\nr}{\operatorname{nr}}

\newcommand{\SO}{\mathrm{SO}}

\newcommand{\SU}{\mathrm{SU}}
\newcommand{\PSU}{\mathrm{PSU}}
\newcommand{\GL}{\mathrm{GL}}
\newcommand{\PGL}{\mathrm{PGL}}
\newcommand{\SL}{\mathrm{SL}}
\newcommand{\PSL}{\mathrm{PSL}}
\newcommand{\Sp}{\mathrm{Sp}}
\newcommand{\GSp}{\mathrm{GSp}}
\newcommand{\Br}{\mathrm{Br}}

\newcommand{\ConstA}{C_1}
\newcommand{\ConstB}{C_2}
\newcommand{\ConstC}{C_3}
\newcommand{\ConstD}{C_4}
\newcommand{\ConstCA}{C_5}
\newcommand{\ConstE}{C_6}
\newcommand{\ConstG}{C_{7}}
\newcommand{\ConstMN}{C_{8}}
\newcommand{\ConstSubtorus}{C_{9}}
\newcommand{\ConstReg}{C_{10}}
\newcommand{\ConstInter}{C_{11}}
\newcommand{\ConstH}{C_{12}}
\newcommand{\ConstK}{C_{13}}

\def\bG{\mathbf{G}}
\def\bH{\mathbf{H}}

\def\bB{\mathbf{B}}
\def\bT{\mathbf{T}}
\def\bC{\mathbf{C}}

\def\rd{\phi}

\newcommand\uG{\underline{G}}
\newcommand\uN{\underline{N}}
\newcommand\uR{\underline{R}}
\newcommand\uH{\underline{H}}
\newcommand\uI{\underline{I}}
\newcommand\uS{\underline{S}}
\newcommand\uT{\underline{T}}

\newcommand\uZ{\underline{Z}}
\newcommand{\et}{\mathrm{\text{\'e}t}}
\newcommand{\Lattice}{\Lambda}
\newcommand{\ReducedLattice}{\mathrm{L}}

\newcommand{\finiteG}{G}

\newcommand{\finiteP}{P}
\newcommand{\finiteRhoSS}{\bar\rho_\ell^{\mathrm{ss}}}
\newcommand{\finiteAlgebraicG}{\underline{G}}
\newcommand{\finiteAlgebraicH}{\underline{H}}
\newcommand{\finiteAlgebraicS}{\underline{S}}

\newcommand{\finiteAlgebraicZ}{\underline{Z}}
\newcommand{\GroupSchemeO}{\cI}
\newcommand{\GroupSchemeZ}{\cH}

\newtheorem{thm}{Theorem}[section]
\newtheorem{theorem}[thm]{Theorem}
\newtheorem{cor}[thm]{Corollary}

\newtheorem{prop}[thm]{Proposition}
\newtheorem{lemma}[thm]{Lemma}

\newtheorem{conj}[thm]{Conjecture}

\newtheorem{remark}[thm]{Remark}
\newtheorem{lem}[thm]{Lemma}

\newtheorem{innercustomthm}{Theorem}
\newenvironment{customthm}[1]
  {\renewcommand\theinnercustomthm{#1}\innercustomthm}
  {\endinnercustomthm}

\begin{document}

\title[Maximality of Galois actions]{Maximality of Galois actions for abelian and hyperk\"ahler  varieties}

\author{Chun Yin Hui}
\email{pslnfq@gmail.com}
\address{
Yau Mathematical Sciences Center\\
Tsinghua University\\
 Haidian District, Beijing 100084\\
 China
}

\author{Michael Larsen}
\email{mjlarsen@indiana.edu}
\address{
Department of Mathematics \\
Indiana University \\
Bloomington, IN 47405, U.S.A.}

\thanks{Chun Yin Hui was supported by the National Research Fund, Luxembourg (cofunded under the Marie Curie Actions of the European Commission (FP7-COFUND)) and China's recruitment program for young professionals.
Michael Larsen is partially supported by the National Science Foundation.}

\begin{abstract}
Let $\{\rho_\ell\}_\ell$ be the system of $\ell$-adic representations arising from the $i$th
$\ell$-adic cohomology of a proper smooth variety $X$ defined over a number field $K$.
Let $\Gamma_\ell$ and $\bG_\ell$ be respectively the image and the algebraic monodromy group of $\rho_\ell$.
We prove that the reductive quotient of $\bG_\ell^\circ$ 
is unramified over every degree $12$ totally ramified extension of $\Q_\ell$ for all sufficiently large $\ell$.
We give a necessary and sufficient condition $(\ast)$ on $\{\rho_\ell\}_\ell$ such that
for all sufficiently large $\ell$, the subgroup $\Gamma_\ell$ is in some sense maximal compact in $\bG_\ell(\Q_\ell)$.
This is used to deduce Galois maximality results for $\ell$-adic representations arising from
abelian varieties (for all $i$) and hyperk\"ahler varieties ($i=2$)
defined over finitely generated fields over $\Q$.
\end{abstract}

\maketitle
\tableofcontents

\section{Introduction}\label{s1}
\subsection{Galois maximality conjecture}
The starting point of this paper is the well known
theorem of Serre \cite[Th\'eor\`eme~2]{Se72} asserting that for every non-CM elliptic curve 
$E$ over a number field $K$, the $\ell$-adic Galois representation
$\Gal_K\to \GL_2(\Z_\ell)$, given by the Galois action on the  $\ell$-adic Tate module of $E$, is surjective for all $\ell$ sufficiently large.
In his 1984--85 Coll\`ege de France course \cite{Se85}, Serre extended this result to abelian varieties $X$ 
with $\End_{\overline K}(X)=\Z$,
when $g$ is odd (or $g\in\{2,6\}$).  For $\ell\gg 0$, the image of $\Gal_K$ in $\Aut(T_\ell(X))$ is then
$\GSp_{2g}(\Z_\ell)$.  The hypothesis on $g$ ensures that
the only semisimple group  admitting an irreducible, minuscule, symplectic representation of degree $2g$ is $\Sp_{2g}$,
and the only representation of this form the standard one (the relevance of this to abelian varieties is due, to the best of our
knowledge, to Ribet \cite{Ri}.)
When
$g = 4m$, for example, $\Sp_{2m} \times \SL_2\times \SL_2$ has an irreducible, minuscule, symplectic representation of degree $2g$,
namely the external tensor product of the natural representations of the three factors, and a new idea is needed.

Wintenberger \cite[Th\'eor\`eme~2]{Wi} succeeded in proving a general result for abelian varieties, 
which can be formulated as follows.
Let $X$ be an abelian variety of dimension $g$ over a number field $K$.
Let $\Gamma_\ell$ be the image of $\Gal_K$ in $ \Aut(T_\ell(X))$ and
$\bG_\ell$ the Zariski closure of $\Gamma_\ell$ in $\GL_{2g}$ over $\Q_\ell$.   Let $\bG_\ell^{\sc}$ denote the universal covering group of
the derived group of the identity component $\bG_\ell^\circ$. Then for all sufficiently large $\ell$,
the group $\Gamma_\ell$ contains the image of a hyperspecial maximal compact 
subgroup of $\bG_\ell^{\sc}(\Q_\ell)$.  A key ingredient in Wintenberger's argument was Falting's proof of Tate's conjecture for abelian varieties \cite[Satz~4]{Fa83},
which guarantees that certain algebraic cycles predicted by Tate's general conjecture exist.

The goal of this paper is to show it is feasible to prove theorems of this type with less powerful information from arithmetic geometry; 
instead of assuming that certain algebraic cycles exist, it is enough to 
assume equality between certain dimensions of $\ell$-adic and (mod $\ell$) Tate cycles.
This not only gives a new proof of Wintenberger's theorem for abelian varieties (where we know the needed algebraic cycles exist, as graphs of endomorphisms)
it gives new Galois maximality results, for instance for hyperk\"ahler varieties, where we do not know it.  Beyond these special cases, 
it offers a general approach to proving maximality of Galois images without first proving a version of the Tate Conjecture or the Mumford-Tate Conjecture.  The price for doing this is harder work on the group theory side.

Let $X$ be a proper smooth variety $X$ defined over a finitely generated subfield $K\subset \C$. 
Let $\overline K$ denote the algebraic closure of $K$ in $\C$ and $X_{\overline K}:= X\times_K\overline K$. 
For a fixed non-negative integer $i$ and a varying 
rational prime $\ell$, each $\ell$-adic \'etale cohomology cohomology group $H^i(X_{\overline K},\Q_\ell)$ is a $\Q_\ell$-vector space 
with a continuous $\Gal_K:=\Gal(\overline K/K)$-action.
Let $n$ be the common dimension of $H^i(X_{\overline K},\Q_\ell)$ for all $\ell$.
We obtain a system of 
$\ell$-adic representations
\begin{equation}\label{system}
\{\rho_\ell:\Gal_K\to \GL_n(\Q_\ell)\}_\ell
\end{equation}
which in the case that $K$ is a number field is (by the main theorem of Deligne \cite{De74}) a \emph{strictly compatible system}
in the sense of Serre \cite[Chapter 1]{Se98}.
The image $\Gamma_\ell := \rho_\ell(\Gal_K)$, is called the \emph{monodromy group} of $\rho_\ell$; it 
is a compact $\ell$-adic Lie subgroup of $\GL_n(\Q_\ell)$. 

The \emph{algebraic monodromy group} of $\rho_\ell$, denoted by $\bG_\ell$,
is defined to be the Zariski closure of $\Gamma_\ell$ in $\GL_{n,\Q_\ell}$.
There exists a finite extension $L/K$ such that $\rho_\ell(\Gal_L)\subset\bG_\ell^\circ(\Q_\ell)$
for all $\ell$ (see \cite[p.6,17]{Se81}, \cite[$\mathsection2.2.3$]{Se85}.)
When $X/K$ is projective, the conjectural theory of motives together with the celebrated conjectures of Hodge, of Tate, and of Mumford-Tate
predict the existence of a common connected reductive $\Q$-form $\bG_\Q$ of $\bG_\ell^\circ$ 
for all $\ell$ (see \cite[$\mathsection3$]{Se94}).
Then Serre's conjectures on maximal motives \cite[11.4, 11.8]{Se94} imply that 
if $\cG$ denotes any extension of $\bG_{\Q}$ to a group scheme over $\Z[1/N]$
for some $N$, the compact subgroup $\rho_\ell(\Gal_L)$ 
is in a suitable sense maximal in $\cG(\Z_\ell)$  
if $\ell$ is sufficiently large.

Denote $\Gamma_\ell\cap\bG_\ell^\circ(\Q_\ell)$ by $\Gamma_\ell^\circ$, 
the derived group of $\bG^\circ_\ell$ by $\bG^{\der}_\ell$,
the intersection $\Gamma_\ell\cap\bG_\ell^{\der}(\Q_\ell)$ by $\Gamma_\ell^{\der}$,
the quotient of $\bG_\ell^\circ$ by its radical by  $\bG_\ell^{\ss}$, and
the image of $\Gamma_\ell^\circ$ under the quotient map
$\bG_\ell^\circ(\Q_\ell)\to \bG_\ell^{\ss}(\Q_\ell)$ by $\Gamma_\ell^{\ss}$.
Since $\bG_\ell^{\ss}$ is connected semisimple, 
it has a universal covering group, which we denote $\bG_\ell^{\sc}$;
we write $\Gamma_\ell^{\sc}$ 
for the preimage of $\Gamma_\ell^{\ss}$ under the map
$\bG_\ell^{\sc}(\Q_\ell)\to\bG_\ell^{\ss}(\Q_\ell)$.
The following statement, due to the second named author, is a weak version of
Serre's weak maximality conjecture with the feature that it can be formulated without
assuming the Mumford-Tate conjecture.  The connections between these conjectures
are explored further in \cite{HL15}. 

\begin{conj}\cite{La95}\label{Laconj}
Let $\{\rho_\ell\}_\ell$ be the system of $\ell$-adic representations
arising from the $i$th $\ell$-adic cohomology of a proper smooth variety $X/K$.
Then the $\ell$-adic Lie group $\Gamma_\ell^{\sc}$ is 
a hyperspecial maximal compact subgroup of $\bG_\ell^{\sc}(\Q_\ell)$
for all sufficiently large $\ell$.
\end{conj}

It is proved in \cite[Theorem~3.17]{La95} that the assertion on $\Gamma_\ell^{\sc}$ holds 
for a density $1$ subset of primes $\ell$.  These primes lie in an infinite union of sets defined by
Chebotarev-type conditions, and there seems no hope of showing by this method 
that this thin set of possible exceptional primes is in fact finite.
In \cite[Theorem~1]{HL16}, we proved Conjecture \ref{Laconj} for ``type A'' Galois representations.
What is special about type A is that semisimple groups of this type contain
no proper semisimple subgroups of equal rank.  For instance, a new idea would be needed to 
rule out possibilities like
$$\Gamma_\ell^{\sc} = \{\gamma\in \Sp_{2n}(\Z_\ell)\mid \bar\gamma\in\Sp_{2m}(\F_\ell)\times \Sp_{2n-2m}(\F_\ell)\},$$
where $\bar \gamma$ denotes (mod $\ell$) reduction.  To rule out this kind of behavior, we introduce a new hypothesis ($\ast$) below.

\subsection{Main results of the paper}
It is convenient to replace $K$ with a finite extension in order that we may assume that $\bG_\ell$ is connected for all $\ell$.
Denote by $\rho_\ell^{\ss}$ the semisimplification of $\rho_\ell$ and by 
$\bG_\ell^{\red}$ the quotient of $\bG_\ell^\circ$ by its unipotent radical. 
The group $\bG_\ell^{\red}$ is also the image of $\bG_\ell^\circ$ 
under the semisimplification of $H^i(X_{\overline K},\Q_\ell)$.
The image of $H^i(X_{\overline K},\Z_\ell)$ is a lattice in  $H^i(X_{\overline K},\Q_\ell)$. Let 
$$\finiteRhoSS:\Gal_K\to \GL_n(\F_\ell)$$ 
be the \emph{semisimple reduction}
of $\rho_\ell$, i.e., the semisimplification of the representation obtained by reducing $\rho_\ell$ modulo $\ell$,
and denote the image of $\finiteRhoSS$ by $\finiteG_\ell$.  By the Brauer-Nesbitt theorem, this does not
depend on the choice of lattice.
For every sufficiently large $\ell$, 
there exists a connected reductive subgroup $\finiteAlgebraicG_\ell$ 
(called \emph{the algebraic envelope}\footnote{The algebraic envelope is denoted by $\bar{\bG}_\ell$ in \cite{Hu15}.  
Here we follow the notation of \cite{Se86a}.}  of $\finiteG_\ell$) 
of $\GL_{n,\F_\ell}$ such that $\finiteG_\ell$ is a subgroup of $\uG_\ell(\F_\ell)$ 
of index bounded above by a constant independent of $\ell$ (see $\mathsection\ref{41}$).

The central result in this paper is Theorem \ref{main} below,
which gives an arithmetic condition equivalent to Conjecture \ref{Laconj}.

\begin{thm}\label{main}
Let $\{\rho_\ell\}_\ell$ be the system of $\ell$-adic representations
arising from the $i$th $\ell$-adic cohomology of a proper smooth variety $X$
defined over a number field $K$ such that $\bG_\ell$ is connected for all $\ell$. 
Then for sufficiently large $\ell$, the subgroup 
$\Gamma_\ell^{\sc}$ is a hyperspecial maximal compact subgroup in $\bG_\ell^{\sc}(\Q_\ell)$ (and $\bG_\ell^{\red}$ is unramified)  
for all sufficiently large $\ell$ if and only if the commutants of
$\Gamma = \rho_\ell^{\ss}(\Gal_K)$ and $G = \finiteRhoSS(\Gal_K)$ on the ambient spaces have the same dimension:
\begin{equation*}
\tag{$\ast$} 
 \dim_{\Q_\ell}(\End_{\Gamma}(\Q_\ell^n))=\dim_{\F_\ell}(\End_{G}(\F_\ell^n)). 
\end{equation*}

\end{thm}

\vspace{.1in}

An even dimensional, projective smooth, simply connected variety $Y$ defined over $K$
is said to be \emph{hyperk$\ddot{a}$hler} if the space of holomorphic $2$-forms $H^0(Y(\C), \Omega_{Y(\C)}^2)$ 
is of dimension one and generated by a form that is non-degenerate everywhere on $Y(\C)$. Examples include Hilbert schemes of points on $K3$ surfaces (including K3 surfaces themselves)
and generalized Kummer varieties.
By using Theorem \ref{main}, we prove the following.

\begin{thm}\label{main2}
Let $\{\rho_\ell\}_\ell$ be the system of $\ell$-adic representations
arising from the $i$th $\ell$-adic cohomology of a proper smooth variety $X$ defined over a subfield $K$ 
of $\C$ that is finitely generated over $\Q$.
For all sufficiently large $\ell$, the group $\Gamma_\ell^{\sc}$ is a hyperspecial maximal compact subgroup
of $\bG_\ell^{\sc}(\Q_\ell)$ and $\bG_\ell^\circ$ is reductive and unramified
under either of the following hypotheses:
\begin{enumerate}[(a)]
\item $X$ is an abelian variety.
\item $X$ is a hyperk$\ddot{a}$hler variety and $i=2$. 
\end{enumerate}
\end{thm}

\begin{remark}
In the special case that Serre originally considered  \cite{Se72}, it takes some additional work to deduce his result from Theorem~\ref{main2} (a).  Namely, one must
translate between the Galois action on the Tate module and the dual action on $H^1$, note that $\bar\rho_\ell = \bar\rho^{\ss}_\ell$ for all $\ell$ sufficiently large,
and check that $\det\rho_\ell$ is surjective.
Serre also presented his $\ell$-adic result as  a consequence of an ``adelic openness'' result: 
the image of $\Gal_K$ in $\GL_2(\hat\Z)$ is open. Theorem \ref{main2} also has an ``adelic'' version, 
which requires some care to state since we do not know that the groups $\bG_\ell$ come from a common
algebraic group over $\Q$.  Details are given in \cite{HL15}.  
\end{remark}

We have already mentioned that Wintenberger proved (a) in the key case that $K$ is a number field and $i=1$.
In the case of K3 surfaces, Theorem~\ref{main2}(b) is due to Cadoret and Moonen \cite[Theorem~B]{CM19}, conditional on the Mumford-Tate conjecture.

\subsection{Ingredients and structure of the paper}
The key ingredient in proving Theorem~\ref{main} is 
the following purely group-theoretic result:

\begin{thm}\label{gpmain}
Let $\bG\subset \GL_{n,\Q_\ell}$ be a connected reductive subgroup,
$\uG\subset \GL_{n,\F_\ell}$ a connected reductive subgroup with $\uG^{\der}$ as derived group and $\uZ$ as connected center, 
$\Gamma$ a closed subgroup of $\bG(\Q_\ell)\cap \GL_n(\Z_\ell)$, and
$\rd\colon \Gamma\to \GL_n(\F_\ell)$ a semisimple continuous representation with $G:=\phi(\Gamma) \subset \uG(\F_\ell)$.
Assume that this data satisfies the following conditions.
\begin{enumerate}[(a)]
\item The subgroup $\Gamma$ is Zariski-dense in $\bG$.
\item There is an equality of semisimple ranks: $\rank \bG^{\der} = \rank \uG^{\der}$.
\item The derived group $\uG^{\der}$ is exponentially generated (see $\mathsection\ref{Nori}$).
\item For all $\gamma\in \Gamma$, the (mod $\ell$) reduction of the characteristic polynomial of $\gamma$ is the characteristic polynomial of $\rd(\gamma)$.
\item The index $[\uG(\F_\ell):G]$ is bounded by $k\in\N$.
\item The formal character of $(\finiteAlgebraicZ,\F_\ell^n)$ is bounded by $N\in\N$, 
where $\finiteAlgebraicZ$ is the connected center of $\uG$ (see $\mathsection\ref{formal}$).
\item Condition ($\ast$) holds for $\Gamma$ and $G$, i.e.,
$$\dim_{\Q_\ell} (\End_{\Gamma} (\Q_\ell^n)) = \dim_{\F_\ell} (\End_{G} (\F_\ell^n)).$$
\end{enumerate}
If $\ell$ is sufficiently large in terms of the data in (a)--(g), then
the reduction representation $\Gamma\hookrightarrow\GL_n(\Z_\ell)\to\GL_n(\F_\ell)$ and $\phi$ are conjugate,
$\Gamma^{\sc}$ is a hyperspecial maximal compact subgroup
of $\bG^{\sc}(\Q_\ell)$, and $\bG^{\der}$ is unramified.
Hypotheses (a)--(f) of Theorem~\ref{gpmain}  suffice to imply that 
$\bG^{\der}$ splits over some finite unramified extension of $\Q_\ell$ and is unramified over 
every degree $12$ totally ramified extension of $\Q_\ell$.
\end{thm}

Once Theorem \ref{gpmain} is established, Theorem \ref{main} follows by 
checking that (after restricting $\rho_\ell$ to some open subgroup $\Gal_L$ of $\Gal_K$ and semisimplifying) 
there exist $n,k,N\in\N$ such that the conditions (a)--(f) in Theorem \ref{gpmain} are verified for the monodromy group
$\Gamma := \rho_\ell(\Gal_L)$, $G := \rho_\ell^{\ss}(\Gal_L)$, $\bG := \bG_\ell^\circ$, and $\uG$ is
the algebraic envelope of $\rho_\ell^{\ss}(\Gal_L)$ if $\ell$ is sufficiently large.
This verification uses the main results in \cite{Hu15}. 
Unfortunately, condition (g) is in a different category, and further progress on Conjecture~\ref{Laconj}
seems to require it.

Theorem \ref{main2}(a) is a consequence of Theorem \ref{main},
a (mod $\ell$) version of the Tate conjecture for abelian varieties for $\ell\gg 0$ proved by Faltings \cite[Theorem~4.2]{FW84} for the condition $(\ast)$
in Theorem \ref{main},
and an $\ell$-independence result of algebraic monodromy groups under specialization \cite[Corollary~2.7]{Hu12}.
Theorem \ref{main2}(b) is a consequence of Theorem \ref{main2}(a), the Kuga-Satake construction,
and Andr\'e's results on motivated cycles \cite{An96a,An96b}.  Indeed, we will see 
that the condition that ($\ast$) holds for all sufficiently large $\ell$ is stable under duals, tensor products, and passage to subrepresentations (see Lemma \ref{hypermorph}).

The last claim of Theorem~\ref{gpmain} is  obtained mainly by Bruhat-Tits theory, 
which determines the possibilities for a connected semisimple group $\bG/\Q_\ell$ whose group of $\Q_\ell$-points
contains a maximal compact subgroup whose total $\ell$-rank (see $\mathsection2.3$) equals the rank of $\bG$.

To indicate the idea behind the rest of Theorem \ref{gpmain}, we consider a particularly favorable case. 
Suppose that $n=2g$ is even, $\Gamma\subset \GSp_{2g}(\Z_\ell)$ is
Zariski-dense in $\bG=\GSp_{2g,\Q_\ell}$.
We remark that not all maximal compact subgroups of $\GSp_{2g}(\Q_\ell)$ are of this form $\GSp_{2g}(\Z_\ell)$;
in general, we can achieve an embedding of $\Gamma$ to this kind only after passing to a (totally ramified)
finite extension of $\Q_\ell$, but we assume this for purposes of illustration.
We further assume that the (mod $\ell$) reduction $\Gamma\to \GL_{2g}(\F_\ell)$ is semisimple, so the reduction map can be identified, after conjugation, with $\phi$.
The goal is to show
that the inclusion $\Gamma' := [\Gamma,\Gamma]\subset\Sp_{2g}(\Z_\ell)$ is an equality 
if $\ell$ is large enough with respect to $n$. 
Reducing (mod $\ell$),  we obtain $G'\subset \Sp_{2g}(\F_\ell)$.

By applying Nori's theory to $G'$ (see $\mathsection\ref{Nori}$), 
we obtain a connected algebraic group
$$\finiteAlgebraicS\subset \Sp_{2g,\F_\ell}$$
such that $G'$ is of bounded index, independent of $\ell$, in $\uS(\F_\ell)$.  
However, \ref{gpmain}(e) 
implies that $G'$ is also of bounded index in $\uG^{\der}(\F_\ell)$, and this, given that both $\uS$ and $\uG^{\der}$ are
generated by additive algebraic groups ($\uS$ by construction, $\uG^{\der}$ by \ref{gpmain}(c))
and therefore connected and also, in some sense, of bounded complexity, implies $\uS = \uG^{\der}$ if $\ell$ is sufficiently large.

By \ref{gpmain}(b), the ranks of $\finiteAlgebraicG^{\der}$ and $\bG^{\der} = \Sp_{2g,\Q_\ell}$ coincide, so both must be $g-1$.
Since $\Gamma$ is Zariski dense in $\GSp_{2g}(\Z_\ell)$, which acts irreducibly on $\Q_\ell^{2g}$, the group
$G$ acts irreducibly on $\F_\ell^{2g}$ by ($\ast$).  It follows that $\uG$ and therefore $\uG^{\der}=\uS$ also act irreducibly.

In the inclusion of connected semisimple groups $\finiteAlgebraicS\subset\Sp_{2g,\F_\ell}$, both 
have the same rank by \ref{gpmain}(b) and the same commutant in $\End_{2g,\F_\ell}$, namely, the scalars. 
In characteristic zero, the equality $\uS = \Sp_{2g,\F_\ell}$ 
would be an immediate consequence of the Borel--de Siebenthal Theorem \cite{BdS49};
this is known to hold also in positive characteristic except for characteristic $2$ and $3$ \cite{Gi}.

Now $\finiteAlgebraicS=\Sp_{2g,\F_\ell}$ is simply connected, so it has no proper subgroups of bounded index as $\ell$ grows without bound.
Thus, for sufficiently large $\ell$, we obtain
\begin{equation}\label{e3}
G'=\finiteAlgebraicS(\F_\ell)=\Sp_{2g}(\F_\ell).
\end{equation}
Finally, a result of Serre \cite[Lemme 1]{Se86a}, subsequently generalized by Vasiu
\cite{Va03}, asserts that any closed subgroup of $\Sp_{2g}(\Z_\ell)$ which maps onto $\Sp_{2g}(\F_\ell)$ is all of $\Sp_{2g}(\Z_\ell)$.
Applying this to $\Gamma'$, we get the theorem in this case.

Various group theoretic technicalities arise in implementing this idea in the general case.  
The condition $(\ast)$ gives a loose comparison between the \emph{reductive} groups 
 $\finiteAlgebraicG$ and $\bG$, while what is needed for
the Borel--de Siebenthal Theorem is a comparison between some \emph{semisimple} groups
$\finiteAlgebraicS$ and $\uI$, where the first comes from $G'$ via Nori's construction 
and the second comes from $\bG^{\der}$ via Bruhat-Tits theory.  
Bruhat-Tits theory works best for simply connected
semisimple groups, but it is also useful for the groups we work with to be subgroups of $\GL_n$.  
Much of the technical work in the paper justifies
moving back and forth between a reductive group, its derived group, and the universal cover of the derived group.

Section $2$ assembles results from group theory that are needed in section $3$, 
including Nori's theory, our theory of $\ell$-dimensions and $\ell$-ranks,
Bruhat-Tits theory, and some results about centralizers, formal characters and regular elements.
Section $3$ proves the purely group theoretic Theorem \ref{gpmain}.
Section $4$ presents the main results on the algebraic envelopes $\finiteAlgebraicG_\ell$
and proves Theorems \ref{main} and \ref{main2}.\\

\section*{Acknowledgments}
The second author gratefully acknowledges conversations with Richard Pink twenty years ago in which we developed an approach to proving Theorem~\ref{main} which is similar
in some important respects to that employed in this paper.  Although there are also major differences between this strategy and the methods of this paper, these discussions undoubtedly
influenced the second author's thinking about the problem.

Both authors would like to express their appreciation to the anonymous referee whose pointed comments and questions on an earlier version of this paper led to a major
restructuring of the paper, strengthening the main result and clarifying (we think) the structure of proof.\\

\begin{center}
\textbf{Conventions for schemes and groups}
\end{center}

The symbol $\ell$ always denotes a rational prime.
Suppose $R\to S$ is a homomorphism of commutative rings with unity and $X$ is a scheme over $\mathrm{Spec}(R)$ (or simply $R$). 
Denote the fiber product $X\times_R S:= X\times_{\mathrm{Spec}(R)}\mathrm{Spec}(S)$ also by $X_S$.

A semisimple algebraic group will always be assumed connected.  A simple algebraic group over $F$ is a semisimple group over $F$ which has no proper, connected, closed, normal subgroup defined over $F$.

In order to keep track of the various kinds of groups that arise in this paper, we use the following system. 
An algebraic group over a field $F$ is always assumed to be a smooth affine group variety over $F$.
The bold letters, e.g., $\bG$, $\mathbf{H}$, $\mathbf{S}$, 
denote  algebraic groups over fields of characteristic zero unless otherwise stated.
The underlined letters, e.g., $\finiteAlgebraicG, \finiteAlgebraicH, \finiteAlgebraicS$, 
always denote  algebraic groups over finite fields. 
Given a homomorphism $f:\finiteAlgebraicG\to \finiteAlgebraicH$, we denote by $f(\finiteAlgebraicG)$
the image of $f$ in $\finiteAlgebraicH$ endowed with the unique structure of reduced closed
subscheme; the induced morphism $\finiteAlgebraicG\to f(\finiteAlgebraicG)$ is assumed to be smooth in this paper.
Group schemes over rings of dimension one (e.g., $\Z$, $\Z_\ell$) are denoted by $\cG$ and $\cH$.
Capital Greek letters denote infinite groups, which are generally $\ell$-adic Lie groups, while capital Roman letters denote finite groups.

Simple complex Lie algebras are denoted by $\mathfrak{g}$ and $\mathfrak{h}$.  We identify such algebras with their
Dynkin diagrams, so instead of saying that $\SL_n(\F_q)$ and $\SU_n(\F_q)$ are both of type $A_{n-1}$, we may say they are of type
$\mathfrak g = \mathfrak{sl}_n$.
Letting $\bG$ denote an algebraic group over a field $F$ of any characteristic, 
and $\Gamma\subset\bG(F)$ a subgroup. The rank of $\bG$, denoted by $\rank\bG$, always means the dimension of a maximal torus of $\bG\times_F\overline F$.
We denote by\\

\begin{tabular}{cl}
$\bG^{\der}$  & the derived group of $\bG$,\\
$\bG^\circ$ & the identity component of $\bG$,\\
$\bG^{\ss}$ & the quotient of $\bG^\circ$ by its radical,\\
$\bG^{\sc}$ & the universal cover of $\bG^{\ss}$,\\
$\bG^{\red}$ & the quotient of $\bG^\circ$ by its unipotent radical,\\
$\dim\bG$ & the dimension of $\bG$ as an $F$-variety,\\
$\rk\mathfrak{g}$ & the rank of the simple Lie algebra $\mathfrak{g}$,\\
$\dim\mathfrak{g}$ & the dimension of $\mathfrak{g}$ as a $\C$-vector space,\\
$\Gamma^{\ss}$ & the image of $\Gamma^\circ:=\Gamma\cap\bG^\circ(F)$ under the quotient map  $\bG^\circ\to \bG^{\ss}$,\\
$\Gamma^{\sc}$ & the preimage of $\Gamma^{\ss}$ under the map $\bG^{\sc}(F)\to \bG^{\ss}(F)$,\\
$M_n(R)$ & the ring of $n$ by $n$ matrices with entries in a ring $R$,\\
$\GL_n(R)$ & the group of units of $M_n(R)$.\\
\end{tabular}

\section{Group theoretic preliminaries}
\subsection{Nori's theory}\label{Nori}
Let $n$ be a natural number and suppose that 
$\ell\geq n$. Let $\finiteG$ be a subgroup of $\GL_n(\F_\ell)$. 
Nori's theory \cite{No87} 
produces a connected $\F_\ell$-algebraic subgroup $\finiteAlgebraicS$ of $\GL_{n,\F_\ell}$ 
that approximates $\finiteG$ if $\ell$ is larger than a constant depending only on $n$. 


Let $\finiteG[\ell]:=\{x\in \finiteG~|~ x^\ell=1\}$.  
The subgroup of $\finiteG$ generated by $\finiteG[\ell]$ is denoted by $\finiteG^+$ and is normal in $G$. 
Define $\mathrm{exp}$ and $\mathrm{log}$ by
$$\mathrm{exp}(x):=\sum_{i=0}^{\ell-1}\frac{x^i}{i!}\hspace{.2in}\mathrm{and}\hspace{.2in}
\mathrm{log}(x):=-\sum_{i=1}^{\ell-1}\frac{(1-x)^i}{i}.$$
Denote by $\finiteAlgebraicS$ the (connected) algebraic subgroup of $\mathrm{GL}_{n,\F_\ell}$, defined over $\mathbb{F}_\ell$, generated by the one-parameter subgroups
\begin{equation}\label{one}
t\mapsto x^t:=\mathrm{exp}(t\cdot\mathrm{log}(x))
\end{equation}
for all $x\in \finiteG[\ell]$. 
The $\F_\ell$-subgroup $\finiteAlgebraicS$ is called the \emph{Nori group} of $G\subset\GL_n(\F_\ell)$.
An algebraic subgroup of $\GL_{n,\F_\ell}$ is said to be \textit{exponentially generated}
if it is generated by the one-parameter subgroups $x^t$ in \eqref{one} for some set of unipotent elements $x\in\GL_n(\F_\ell)$.
 Since $\finiteAlgebraicS$ is exponentially generated, $\finiteAlgebraicS$ is an extension of a semisimple group by a unipotent group \cite[$\mathsection3$]{No87}. 
If $y\in M_n(\F_\ell)$  commutes with $x$, then it also commutes with $\log x$ and therefore with the algebraic group $x^t$.  Thus, 
\begin{equation}
\label{plus-cent}
Z_{\finiteG^+}(M_n(\F_\ell)) = Z_{\finiteAlgebraicS(\F_\ell)}(M_n(\F_\ell)) = Z_{\finiteAlgebraicS(\overline\F_\ell)}(M_n(\overline\F_\ell))\cap M_n(\F_\ell).
\end{equation}
The following theorem approximates $\finiteG^+$ by $\finiteAlgebraicS(\F_\ell)$.

\begin{theorem}
 \label{NTB} 
There is a constant $\ConstA(n)$ depending only on $n$ such that if $\ell>\ConstA(n)$ 
and $\finiteG$ is a subgroup of $\mathrm{GL}_n(\mathbb{F}_\ell)$, then the following assertions hold.
\begin{enumerate}[(i)]
\item If $\finiteAlgebraicS$ is the Nori group of $\finiteG$, then $\finiteG^+=\finiteAlgebraicS(\mathbb{F}_\ell)^+$.
\item If $\uS$ is a Nori group, the quotient
$\finiteAlgebraicS(\mathbb{F}_\ell)/\finiteAlgebraicS(\mathbb{F}_\ell)^+$ is a commutative group of order $\leq 2^{n-1}$.
\item If $\uG\subset\GL_{n,\F_\ell}$ is exponentially generated, then the Nori group of $G = \uG(\F_\ell)$ is $\uG$.
\end{enumerate}
\end{theorem}

\begin{proof}
This is due to Nori: parts (i) and (iii) from \cite[Theorem B]{No87} and part (ii) from \cite[Remark 3.6]{No87}.
\end{proof}

A theorem of Jordan \cite{Jo78} says that every finite subgroup $\finiteG$ of $\GL_n(\C)$ has an abelian subgroup $Z$ such that 
the index $[\finiteG: Z]$ is bounded by a constant depending only on $n$. The following theorem is a variant of Jordan's theorem 
in positive characteristic.

\begin{thm}\cite[Theorem C]{No87} \label{NTC}
Let $\finiteG$ be a subgroup of $\GL_n(\F)$, where $\F$ is a finite field of characteristic $\ell\geq n$. Then $\finiteG$ has a commutative subgroup $Z$ of prime to $\ell$ order such that $Z\cdot\finiteG^+$ is normal in $\finiteG$ and 
$$[\finiteG:  Z\cdot\finiteG^+]\leq \ConstB(n),$$
where $\ConstB(n)$ is a constant depending only on $n$ (and not on $\F$, $\ell$, $\finiteG$).
\end{thm}

Note that the statement of \cite[Theorem C]{No87} does not explicitly assert that the order of $ Z$ is prime to $\ell$, but this fact is stated in the introduction to the paper
and is immediate from the construction of $Z$ \cite[p.~270]{No87}.
Nori's theory does a good job of algebraically approximating the ``semisimple'' and ``unipotent'' parts of a finite subgroup of of $\GL_n(\F_\ell)$, but not the toric part;
in general, there is no reason to expect, for large $\ell$, that $Z$ will be well approximated by the $\F_\ell$-points of \emph{any} torus in $\GL_{n,\F_\ell}$.  For Theorem~\ref{gpmain}, we hypothesize that it can be well approximated, moreover, by a torus whose complexity is bounded in a sense to be made precise in \S2.6.    For images of (mod $\ell$) Galois representations arising from cohomology of a given projective non-singular variety, we will see that this additional hypothesis holds.

We have the following result due to Serre for $\finiteAlgebraicS$ if $\finiteG$ acts semisimply on the ambient space.

\begin{prop}
\label{serre}
 (see e.g., \label{ssaction}\cite[Proposition 2.1.2]{Hu15})
Suppose $\finiteG$ acts semisimply on $\F_\ell^n$. There is a constant $\ConstC(n)$ depending only on $n$ such that  if $\ell>\ConstC(n)$, then the following assertions hold.
\begin{enumerate}
\item[(i)] The Nori group $\finiteAlgebraicS$ is a semisimple $\mathbb{F}_\ell$-subgroup of $\mathrm{GL}_{n,\mathbb{F}_\ell}$.
\item[(ii)] The representation $\finiteAlgebraicS\to \mathrm{GL}_{n,\mathbb{F}_\ell}$ is semisimple.
\end{enumerate}
\end{prop}

\subsection{Galois cohomology}\label{s2.2}

We begin with an estimate in Galois cohomology.

\begin{prop}\label{H1}
For $k\in\N$ there exists a constant $\ConstD(k)$ depending only on $k$ such that if
$F$ is a finite extension of $\Q_\ell$ with  $\ell > k$ 
and  $\bC$ is a finite commutative group scheme over $F$ with $|\mathbf{C}(\overline F)| \le k$, then 
$$|H^1(F,\mathbf{C}(\overline F))| \le \ConstD(k).$$
\end{prop}

\begin{proof}
Consider the exact sequence of abelian groups
$$1\to H^1(\Gal(L/F),\mathbf{C}(\overline F))\to H^1(F,\mathbf{C}(\overline F))\to H^1(L,\mathbf{C}(\overline F))$$
by the inflation-restriction exact sequence, where $\Gal_L$ acts trivially on $\mathbf{C}(\overline F)$
and $[L:F]\leq k!$.
Since the size of $H^1(\Gal(L/F),\mathbf{C}(\overline F))$
is bounded above by some constant depending only on $k$,
it suffices to bound $H^1(L,\mathbf{C}(\overline F))$. 
Let $S$ be the set of abelian extensions 
of $L$ of degree bounded above by $k$. For every element $\phi$ of 
$$H^1(L,\mathbf{C}(\overline F))\cong\mathrm{Hom}(\Gal_L,\mathbf{C}(\overline F))\cong\mathrm{Hom}(\Gal_L^{\mathrm{ab}},\mathbf{C}(\overline F)),$$
$\ker\phi$ corresponds to an element of $S$. 
Since $|\mathbf{C}(\overline F)|\leq k$, we have
$$|\mathrm{Hom}(\Gal_L^{\mathrm{ab}},\mathbf{C}(\overline F))|\leq |S|\cdot k!.$$
Let $\F_q$ be the residue field of $L$.
By local class field theory, $S$ corresponds to the set of 
open subgroups $U$ of 
$$L^*=O_L^*\times\Z=\text{pro-}\ell\times\F_q^*\times\Z$$
such that $[L^*:U]\leq k$. Hence, the possibilities of $U$ is bounded above by some constant depending only on $k$ if $\ell>k$.
\end{proof}

\begin{cor}\label{index}
Let $F$ be a finite extension of $\Q_\ell$ and $\alpha:\bG\to\mathbf{H}$
 a central isogeny of degree $\le k$ of connected reductive groups over $F$.
If $\Gamma$  a subgroup of $\mathbf{H}(F)$, then the quotient
$$\Gamma/\alpha(\alpha^{-1}(\Gamma)\cap\bG(F))$$
is an abelian group with size bounded above by $\ConstD(k)$ if $\ell>k$.
\end{cor}

\begin{proof}
Consider the long exact sequence in Galois cohomology
$$1\to\mathbf{C}(F)\to\bG(F)\to\mathbf{H}(F)\to H^1(F,\mathbf{C}(\overline F))\to\cdots,$$
where $\bC := \ker \alpha$.  
The claim is an immediate consequence of Proposition \ref{H1}.
\end{proof}

\begin{prop}
\label{inner}
Let $\bG$ be a simply connected semisimple group that is an inner twist 
of a split group over a finite extension $F$ of $\Q_\ell$ and 
$d$ denote the order of the center of $\bG(\overline F)$.
For every finite extension $F'$ of $F$
of degree divisible by $d$, the group $\bG$ splits over $F'$.
\end{prop}

\begin{proof}
Let $\bG_0$ be the split form of $\bG$, and let $\bC_0$ denote the center of $\bG_0$.
Now $\bG$ is the twist of $\bG_0$ by a class in $H^1(F,\bG_0(\overline F)/\bC_0(\overline F))$.  The non-abelian cohomology sequence of the central extension
$$1\to \bC_0(\overline F)\to \bG_0(\overline F) \to \bG_0(\overline F)/\bC_0(\overline F)\to 1$$
gives an  exact sequence
$$H^1(F,\bG_0(\overline F)) \to H^1(F,\bG_0(\overline F)/\bC_0(\overline F)) \to H^2(F,\bC_0(\overline F)).$$
%
Thus, since $H^1(F,\bG_0(\overline F)) = 0$ by \cite[Th\'eor\`eme~4.7]{BT}, it suffices to prove that $d$ divides $[F':F]$ implies that
$$H^2(F,\bC_0(\overline F)) \to H^2(F',\bC_0(\overline F))$$
is the zero map.  As $\bG_0$ is split, $\bC_0$ is a product of groups of the form $\mu_n$ where $n$ divides $d$.  Thus, it suffices to prove that every class in $\Br(F)_n$ lies in $\ker (\Br(F)\to \Br(F'))$ for 
every extension $F'/F$ such that $d$ divides $[F':F]$.  This follows from the fact \cite[XIII~Proposition~7]{SeLF} that at the level of invariants, the map $\Br(F)\to \Br(F')$ is just multiplication by $[F':F]$.
\end{proof}

\subsection{$\ell$-dimension and $\ell$-ranks}\label{s2.3}
In this subsection, we review the definitions of the $\ell$-dimension and the $\ell$-ranks 
(i.e., the total $\ell$-ranks and the $\mathfrak{h}$-type $\ell$-rank for varying simple Lie type $\mathfrak{h}$) 
of finite groups and profinite groups with open pro-solvable subgroups \cite{Hu15,HL16} and state the results relating the dimension 
and the ranks of algebraic group $\finiteAlgebraicG/\F_q$ to respectively the $\ell$-dimension and the $\ell$-ranks of $\finiteAlgebraicG(\F_q)$ \cite{HL16}. 

\subsubsection{}
Let $\finiteG$ be a finite simple group of Lie type in characteristic $\ell\geq 5$. 
The condition on $\ell$ rules out the possibility of Suzuki or Ree groups,
so there exists an (unique) adjoint simple group $\finiteAlgebraicG/\F_{\ell^f}$ so that
$$\finiteG=[\finiteAlgebraicG(\F_{\ell^f}),\finiteAlgebraicG(\F_{\ell^f})] = \im\bigl(\uG^{\sc}(\F_{\ell^f})\to \uG(\F_{\ell^f})\bigr).$$
We define the $\ell$-dimension of $\finiteG$ to be 
$$\dim_\ell\finiteG:=f\cdot\dim\finiteAlgebraicG.$$
%

Let $\mathfrak{g}$ denote the unique simple complex
Lie algebra whose root system is a factor of the root system of $\uG_{\overline\F_\ell}$.
If $\mathfrak{h}$ is a simple complex Lie algebra, the $\mathfrak{h}$-type $\ell$-rank of $G$ is
\begin{equation*}
\rk^{\mathfrak{h}}_\ell\finiteG :=\left\{ \begin{array}{lll}
 f\cdot\rank \finiteAlgebraicG &\mbox{if}\hspace{.1in} \mathfrak{h}=\mathfrak{g},\\
 0 &\mbox{otherwise.}
\end{array}\right.
\end{equation*}

\vspace{.1in}
For example, $\PSL_n(\F_{\ell^f})$ (resp. $\PSU_n(\F_{\ell^f})$) has $f(n^2-1)$ as the $\ell$-dimension, 
$f(n-1)$ as $A_{n-1}$-type $\ell$-rank.

For simple groups which are not of Lie type in characteristic $\ell$ 
(including simple groups of order less than $\ell$ and abelian simple groups like $\Z/\ell\Z$), 
we define the $\ell$-dimension and $\mathfrak{h}$-type $\ell$-rank to be zero.  
We extend the definitions
to arbitrary finite groups $G$ by defining the $\ell$-dimension (resp. $\mathfrak{h}$-type $\ell$-rank) 
to be the sum of the $\ell$-dimensions (resp. $\mathfrak{h}$-type $\ell$-ranks)
of its composition factors.
We define the total $\ell$-rank of $G$ to be
$$\rk_\ell G := \sum_{\mathfrak{h}} \rk_\ell^{\mathfrak{h}},$$
where the sum is taken over all simple complex Lie algebras.

This makes it clear that $\dim_\ell$, $\rk_\ell^{\mathfrak{h}}$, and $\rk_\ell$ are additive on short exact sequences of groups.  In particular, the $\ell$-dimension and the total $\ell$-rank of
every solvable finite group are zero, and neither passing to a central extension nor to the derived group affects the $\ell$-dimensional or any $\ell$-rank.  
For instance the $\ell$-dimension and $\ell$-rank of $\GL_n(\F_{\ell^f})$, $\PGL_n(\F_{\ell^f})$ and $\PSL_n(\F_{\ell^f})$ are all the same.

Our basic results on $\dim_\ell$, $\rk_\ell^{\mathfrak{h}}$, and $\rk_\ell$ of finite groups are the following.

\begin{lem}
\label{eqrank}
For $k\in\N$ there exists a constant $\ConstCA(k)$ such that if $\ell > \ConstCA(k)$
and $H$ is a subgroup of $G$ of index $\le k$, then the $\ell$-dimension and $\ell$-ranks of $G$ and $H$ are the same.
\end{lem}

\begin{proof}
At the cost of replacing $k$ by $k!$, we may assume $H$ is normal in $G$.  If $\ell$ is large enough
the $\ell$-dimension and $\ell$-ranks of $G/H$ are $0$, and the lemma follows by additivity.
\end{proof}

\begin{prop}\label{DR0}
Let $\finiteG$ be a subgroup of $\GL_n(\F_\ell)$ and $\finiteAlgebraicS$ the Nori group of $\finiteG$. There exists a constant $\ConstE(n)$ depending only on $n$ such that if $\ell>\ConstE(n)$,
then the $\ell$-dimension and the $\ell$-ranks of $\finiteG$ and $\finiteAlgebraicS(\F_\ell)$ are identical.  
\end{prop}

\begin{proof}
The assertion follows directly from Theorems \ref{NTB} and \ref{NTC} and Lemma \ref{eqrank}.
\end{proof}

\begin{prop}\cite[Proposition 4]{HL16}\footnote{The rank of an algebraic group $\bG/F$ in \cite{HL16}
is defined to be the usual rank of $\bG^{\ss}\times_F\overline F$, see \cite[$\mathsection2$]{HL16}.} \label{DR1}
Let $\finiteAlgebraicG$ be a connected algebraic group over $\mathbb{F}_{\ell^f}$ with $\ell\geq5$. 
The composition factors of $\finiteAlgebraicG(\mathbb{F}_{\ell^f})$ are cyclic groups and 
finite simple groups of Lie type in characteristic $\ell$. 
Moreover, let $m_\mathfrak{g}$ be the number of factors of 
$\finiteAlgebraicG^{\sc}_{\overline{\F}_\ell}$ of type $\mathfrak{g}$. Then the following equations hold:
\begin{enumerate}
\item[(i)] $\rk_\ell^\mathfrak{g}(\finiteAlgebraicG(\F_{\ell^f}))=m_\mathfrak{g}f\cdot\rk\mathfrak{g}$;
\item[(ii)] $\rk_\ell(\finiteAlgebraicG(\F_{\ell^f}))=f\cdot\rank\finiteAlgebraicG^{\ss}$ ;
\item[(iii)] $\dim_\ell (\finiteAlgebraicG(\F_{\ell^f}))=f\sum_{\mathfrak{g}}(m_{\mathfrak{g}}\cdot \dim\mathfrak{g})
= f\cdot\dim \finiteAlgebraicG^{\ss}$.
\end{enumerate}
\end{prop}

\subsubsection{}\label{s2.3.2}
Let $F$ be a finite extension of $\Q_\ell$ with the ring of integers $O_F$ and the residue field $\F_q$. The definitions above are extended to certain infinite profinite groups, including compact subgroups of $\GL_n(F)$, as follows.
If $\Gamma$ is a finitely generated profinite group which contains an open pro-solvable subgroup, we define
$$\dim_\ell\Gamma := \dim_\ell(\Gamma/\Delta),\hspace{.1in}
\rk_\ell^{\mathfrak{h}} \Gamma := \rk_\ell^{\mathfrak{h}} (\Gamma/\Delta),\hspace{.1in}\mathrm{and}\hspace{.1in}
\rk_\ell \Gamma := \rk_\ell (\Gamma/\Delta)$$
for any normal, pro-$\ell$, open subgroup $\Delta$ of $\Gamma$. 
The $\ell$-dimension and $\ell$-rank of every pro-$\ell$ group is zero (so, in particular, the $\ell$-dimension of an $\ell$-adic Lie group can be strictly smaller than its dimension in the sense of $\ell$-adic manifolds.)
By additivity,
$$\dim_\ell\Gamma = \dim_\ell\finiteG,\hspace{.1in}
\rk_\ell^{\mathfrak{h}} \Gamma = \rk_\ell^{\mathfrak{h}} \finiteG,\hspace{.1in}\mathrm{and}\hspace{.1in}
\rk_\ell \Gamma = \rk_\ell \finiteG,$$ 
where $\finiteG$
denotes the image in $\GL_n(\F_q)$ under the reduction of $\Gamma$ with
respect to an $O_F$-lattice in $F^n$ stabilized by $\Gamma$.  
If $\Gamma$ is a compact subgroup of $\GL_n(F)$ and $\Delta$ is a closed normal subgroup, then
$$\dim_\ell \Gamma  
= \dim_\ell \Delta + \dim_\ell (\Gamma/\Delta),$$
$$\rk_\ell^{\mathfrak{h}} \Gamma  
= \rk_\ell^{\mathfrak{h}} \Delta + \rk_\ell^{\mathfrak{h}} (\Gamma/\Delta),$$
and
$$\rk_\ell\Gamma=\rk_\ell\Delta+\rk_\ell(\Gamma/\Delta).$$






\begin{lemma}
\label{subgroup}
Let $\Gamma\subseteq \Pi$ be compact subgroups of $\GL_n(\Q_\ell)$ (resp. $\GL_n(\F_\ell)$).  
There exists a constant $\ConstG(n)$ depending only on $n$ such that if $\ell>\ConstG(n)$, then
\begin{align*}
\rk_\ell\Gamma&\le \rk_\ell \Pi,\\
\dim_\ell\Gamma&\le \dim_\ell \Pi.
\end{align*}
\end{lemma}

\begin{proof}
Fix a  $\Pi$-stable lattice $\Lambda$ in $\Q_\ell^n$.
By the fact that the $\ell$-ranks of pro-solvable groups are zero, 
it suffices to prove the same inequality for the finite groups 
$\finiteG\subseteq \finiteP\subseteq \GL_n(\F_\ell)$ obtained by reducing modulo  $\Lambda$. 
The Nori group of $\finiteG$ is generated by a subset of the 
collection of unipotent groups generating the Nori group of $\finiteP$ and is therefore a closed subgroup of that algebraic group.  Both dimension and semisimple rank of a subgroup of any algebraic group are less than or equal to those of the ambient group, so the lemma follows from Propositions~\ref{DR0} and \ref{DR1}.
\end{proof}

\subsection{Bruhat-Tits theory}\label{s2.4}
We briefly recall some basic facts from  Bruhat-Tits theory, mainly from \cite{Ti79}.  
The main goal of this subsection is Theorem \ref{compare}.

\subsubsection{}\label{241}
Let $F$ be a finite extension of $\Q_\ell$ with residue field $\F_q$ 
and $\bG$ a connected, semisimple algebraic group defined over $F$. 
The Bruhat-Tits building $\mathscr{B}(\bG,F)$ is a 
polysimplicial complex \cite[2.2.1]{Ti79}, endowed with a $\bG(F)$-action that is linear 
on each facet.
If $F'$ is a finite extension of $F$, then there is a corresponding continuous injection of buildings
$$\iota_{F',F}:\mathscr{B}(\bG,F)\to\mathscr{B}(\bG,F'),$$
which is equivariant with respect to $\bG(F)\subset \bG(F')$ and maps vertices of 
$\mathscr{B}(\bG,F)$ to vertices of $\mathscr{B}(\bG,F')$.
If $F\subset F'\subset F''$ are finite extensions of fields, then
$$\iota_{F'',F'}\circ\iota_{F',F}=\iota_{F'',F}.$$

For every point $x\in\mathscr{B}(\bG,F)$, the stabilizer $\bG(F)^x$
is a compact subgroup of $\bG(F)$.
There exists a smooth affine group scheme $\cG_x$ over the ring of integers $O_F$ of $F$
and an isomorphism $i$ from the generic fiber of $\cG_x$ to $\bG$ 
such that $i(\cG_x(O_F)) = \bG(F)^x$ and  
if $F'$ is a finite unramified extension of $F$, then
$$i(\cG_x(O_{F'})) = \bG(F')^{\iota_{F',F}(x)}$$
\cite[3.4.1]{Ti79}. If the special fiber $\cG_{x,\F_q}$ of $\cG_x$ is reductive, we say that $x$ is \emph{hyperspecial}
and $\bG(F)^x$ is a \emph{hyperspecial maximal compact subgroup} (or simply \emph{hyperspecial}) of $\bG(F)$ \cite[3.8.1]{Ti79}.

Every maximal compact subgroup of $\bG(F)$ is the stabilizer $\bG(F)^x$ of 
a point  $x\in \mathscr{B}(\bG,F)$ by \cite[3.2]{Ti79}. 
We may always take $x$ to be the centroid of some facet.
Moreover, if $\bG$ is in addition simply connected, then
$x$ is a vertex \cite[3.2]{Ti79} and the special fiber $\cG_{x,\F_q}$ is connected \cite[3.5.2]{Ti79}.

\subsubsection{}\label{242}
Let $F^{\nr}$ be the maximal unramified extension of $F$ in $\overline F$.
The group $\bG$ determines a map of $\Gal(F^{\nr}/F)$-diagrams: 
the \emph{relative local Dynkin diagram} $\Delta_F$ (i.e., the local Dynkin diagram of $\bG/F$) with trivial $\Gal(F^{\nr}/F)$-action,
the \emph{absolute local Dynkin diagram} $\Delta_{F^{\nr}}$ (i.e., the local Dynkin diagram of $\bG_{F^{\nr}}/F^{\nr}$) 
with an action of $\Gal(F^{\nr}/F)$,
and a $\Gal(F^{\nr}/F)$-map $\Delta_{F^{\nr}}\to \Delta_F$ \cite[1.11]{Ti79}. 
The Dynkin diagram of $\cG_{x,\overline\F_\ell}^{\red}$ (the reductive quotient of $\cG_{x,\overline\F_\ell}$, see conventions for groups) 
can be constructed by 
deleting from $\Delta_{F^{\nr}}$ all the vertices (together with all edges connected to them) 
mapping to the vertices in $\Delta_F$ associated to $x$.
Moreover, if the minimal facet containing $x$ is a chamber (for example when $\bG/F$ is anisotropic, in which case $\Delta_F$ is empty), 
$\cG_{x,\F_q}^{\red}$ is a torus  \cite[3.5.2]{Ti79}. 

A semisimple group $\bG$ over a local field
 $F$ is \emph{unramified} if $\bG$ has a Borel subgroup over $F$
and $\bG$ splits over an unramified extension of $F$.
The group $\bG$ is unramified if and only if $\mathscr{B}(\bG,F)$ has a hyperspecial point 
(see \cite[$\mathsection1.10.2$]{Ti79} for the ``only if" part and \cite[Corollary 5.2.14]{Co14} for the ``if" part).
And the latter one is equivalent to the local Dynkin diagram 
$\Delta_F$ has a hyperspecial vertex \cite[1.9,1.10]{Ti79}.

\subsubsection{}\label{243} The main theorem of this subsection is as follows.

\begin{thm}
\label{compare}
Let $F$ be a finite extension of $\Q_\ell$ with residue degree $f:=[\F_q:\F_\ell]$  and  $\ell\geq 5$.
Let $\bG$ be a  semisimple group of rank $r$ over $F$ and $\Pi$ a maximal compact subgroup of $\bG(F)$.
Then the following assertions hold.

\begin{enumerate}[(i)]
\item The total $\ell$-rank of $\Pi$ is at most $fr$.
\item If  $\rk_\ell\Pi=fr$, then $\bG$ splits over a finite unramified extension of $F$.
\item If  $\rk_\ell\Pi=fr$, then $\bG$ is unramified over every degree $12$ totally ramified extension $F^{\mathrm{t}}/F$.
\item If  $\rk_\ell\Pi=fr$, then there exist a totally ramified extension $F'/F$ 
and a hyperspecial maximal compact subgroup $\Omega\subset\bG(F')$ such that
$\Pi\subset\Omega$.
\end{enumerate}

\end{thm}

\begin{proof}
Since $\Pi^{\sc}$ is maximal compact in $\bG^{\sc}(F)$ and the total $\ell$-rank of $\Pi$ and $\Pi^{\sc}$ are equal, 
we may assume $\bG$ is simply connected.
It therefore factors as a product of groups $\bG_i$ which are simply connected and  simple.
Let $x$ denote a vertex of the building $\mathscr{B}(\bG,F)$ stabilized by $\Pi$
and $\cG_x$ the smooth affine group scheme over $O_F$ in $\mathsection\ref{241}$. 
The building of $\bG$ is the product of the buildings of the $\bG_i$ \cite[\S2.1]{Ti79}, so the vertex $x = (x_1,\ldots,x_k)$, and
$$\cG_x = \prod_i (\cG_i)_{x_i}.$$
As rank is additive in products, it suffices to prove the theorem in the simple case.

Thus, there exist a finite extension $F'/F$ and an absolutely simple group $\bG'/F'$ such that
$\bG = \Res_{F'/F} \bG'$.  Then $\Pi$ is a maximal compact subgroup of $\bG'(F') = \bG(F)$.
Denote the rank of $\bG'$ by $r'$, 
the order of the residue field of $F'$ by $\ell^{f'}$, 
and the ramification degree of $F'/F$ by $e$. 
Since we have
$$r = [F':F] r' = e (f'/f) r',$$
the inequality $\rk_\ell \Pi \le f'r'$ implies $\rk_\ell \Pi\le fr$, with strict inequality if $e>1$.
Let $F^{\mathrm{t}}/F$ be the totally ramified extension described in (iii).
If $e=1$, then it follows that 
\begin{itemize}
\item $F'/F$ is unramified,
\item 
the composition $F'F^{\mathrm{t}}$ is totally ramified over $F'$ of degree $12$, 
\item and $\bG_{F^{\mathrm{t}}}=(\Res_{F'/F}\bG')\times_F F^{\mathrm{t}} = \Res_{F'F^{\mathrm{t}}/F^{\mathrm{t}}}(\bG'\times_{F'} F'F^{\mathrm{t}})$.
\end{itemize}
Hence, if $\bG'$ splits over an unramified extension $F''$ of $F'$, then $\bG$ (resp. $\bG_{F^{\mathrm{t}}}$) 
also splits over $F''$ (resp. $F'' F^{\mathrm{t}}$), which is unramified over $F$ (resp. $F^{\mathrm{t}}$).  
If $\bG'$ is quasi-split over $F'F^{\mathrm{t}}$, it has a Borel subgroup $\bB'$ defined over $F'F^{\mathrm{t}}$,
and the restriction of scalars $\Res_{F'F^{\mathrm{t}}/F^{\mathrm{t}}} \bB'$ is a Borel subgroup of $\bG_{F^{\mathrm{t}}}$.  
Thus, if (i)--(iii) hold for $(\bG',F')$, they hold for $(\bG,F)$, and without loss of generality,
we may assume $\bG$ is absolutely simple.

As the kernel of $\cG_x(O_F)\to \cG_x(\F_{q})$ is pro-$\ell$, the total $\ell$-ranks of $\Pi$ and $\cG_x(\F_{q})$ are equal.
Since $\bG$ is simply connected, $\cG_x(\F_{q})$ is the group of $\F_{q}$-points of an algebraic group which is the extension of the reductive group $\cG_{x,\F_q}^{\red}$
by a unipotent group.  Thus, $\rk_\ell (\cG_x(\F_{q}))$ is $f$ times $\rank\cG_{x,\F_q}^{\ss}$, the semisimple rank of $\cG_{x,\F_q}^{\red}$.  
We claim that this is less than or equal to $fr$, or equivalently, 
\begin{equation}\label{BTineq}
\rank\cG_{x,\F_q}^{\ss}\leq r,
\end{equation}
and with equality only if
$\bG$ splits over an unramified extension and has a Borel over $F^{\mathrm{t}}$ in (iii).

Since $\rank\cG_{x,\F_q}^{\red}$ and the relative rank of $\bG_{F^{\nr}}$ 
(the rank of a maximal $F^{\nr}$-split torus of $\bG_{F^{\nr}}$) are equal \cite[3.5]{Ti79}, 
the inequality \eqref{BTineq} holds in general, which is assertion (i),
and the equality holds only if $\bG$ splits over $F^{\nr}$, which is assertion (ii).
Let $\bG^{\mathrm{sp}}/F$ be a split form of $\bG$. 
By definition, the number of vertices in the absolute local Dynkin diagram $\Delta_{F^{\nr}}$ of $\bG$
is one greater than the relative rank of $\bG_{F^{\nr}}$.
If the equality in \eqref{BTineq} holds, then it follows by (ii) and $\mathsection\ref{242}$ that
\begin{enumerate}[(A)]
\item $\Delta_{F^{\nr}}$ coincides with the (relative) local Dynkin diagram of $\bG^{\mathrm{sp}}/F$,
\item $\Delta_{F^{\nr}}$ contains at least one $\Gal(F^{\nr}/F)$-stable vertex, 
\item and $\bG$ is not anisotropic, i.e., $\Delta_F\neq\emptyset$. 
\end{enumerate}

To list the cases when the conditions (A),(B),(C) hold, we consult the tables \cite[\S\S 4.2--4.3]{Ti79};  
the possible types are all split types in \cite[\S 4.2]{Ti79} together with the following possibilities in \cite[\S 4.3]{Ti79}:
$${}^2A'_n,\ {}^2B_s,\ {}^2C_{2m},\ {}^2D_t,\ {}^2D'_t,\ {}^2 D''_{2s},\ {}^3D_4,\ {}^4D_{2n},\ {}^2E_6,\ {}^3E_6,\ {}^2E_7,$$
where $m\geq1$, $n\geq2$, $s\geq3$, and $t\geq 4$ are integers.
From the tables, every split type has a hyperspecial vertex in $\Delta_F$ and is thus unramified.  Similarly,
the groups ${}^2A'_n$, ${}^2D_t$, ${}^3D_4$, and ${}^2E_6$ also have hyperspecial vertices in $\Delta_F$ and are therefore unramified.

The cases ${}^2C_{2m}$ and ${}^2E_7$ are inner forms of  split groups of types $C_{2m}$ and $E_7$ respectively by (A)
and the fact that $C_{2m}$ and $E_7$ have no non-trivial outer automorphisms.
They therefore split over every even degree extension by
Proposition~\ref{inner} because their $\overline F$-centers are of order $2$. 

Similarly, ${}^3E_6$ is an inner form of a split group of type $E_6$ by (A) 
and the fact that its index ${}^1E_{6,2}^{16}$ has presuperscript $1$ \cite[$\mathsection4.3$]{Ti79} 
(meaning that the image of $\Gal(F^{\nr}/F)\to \Aut(\Delta_{F^{\nr}})$ is of order $1$ \cite[Table II: Indices]{Ti66}).
Thus it splits over every extension whose degree is divisible by $3$ 
by Proposition~\ref{inner} since its $\overline F$-center is of order $3$.

To see that in the remaining cases
${}^2B_s, {}^2D'_t, {}^2 D''_{2s}, {}^4D_{2n}$, the group
$\bG$ becomes unramified over every $F^{\mathrm{t}}$ in (iii), we examine the explicit descriptions 
\cite[\S 4.4]{Ti79} which classifies every every central isogeny class of absolutely simple
groups over $F$. 
The quaternionic orthogonal groups 
${}^2 D''_{2s}$ and ${}^4D_{2n}$ become ordinary orthogonal groups after passage to any ramified quadratic extension of $F$ since each such extension splits the
quaternion algebra over $F$.  This leaves the cases of orthogonal groups of quadratic forms including ${}^2B_s$ and ${}^2D'_t$.

By passing to any ramified quadratic extension $F'/F$, we may assume that the form $Q$ defining $\bG$ is $u_1 x_1^2+\cdots+u_n x_n^2$, where the
$u_i$ are units in $O_{F'}$.  We claim that 
there exists a plane hyperbolic with respect to $Q$ contained in the $3$-dimensional 
locus $x_4=x_5=\cdots=x_n=0$.  Indeed, the quadratic form $u_1 x_1^2+u_2 x_2^2 + u_3 x_3^2$ defines a form of $\SO(3)$; 
the space $(F')^3$ of triples $(x_1,x_2,x_3)$ contains a hyperbolic plane if and only if this form is split, i.e., if and only if there is a non-zero isotropic vector.
Non-trivial solution of the equation $\bar u_1 x_1^2 + \bar u_2 x_2^2 + \bar u_3 x_3^2 =0$ 
over the residue field $\F_{q'}$ of $F'$ 
exists by Chevalley-Warning, which lifts  by Hensel's lemma
to a non-zero isotropic vector in $(F')^3$.  
It follows by induction that $Q$ defines a quadratic form of Witt index $n' \ge n/2-1$. 
Hence, $\bG_{F'}$ can only be $D_{n'},B_{n'}, {}^2D_{n'+1}$ by \cite[\S 4.4]{Ti79} and (A)
and is thus unramified as its relative local Dynkin diagram has a hyperspecial vertex \cite[\S\S 4.2--4.3]{Ti79}.

Since a finite totally ramified extension $F'$ of $F$ contains subextensions
of all possible degrees dividing $[F':F]$,
we conclude that if $F^{\mathrm{t}}$ is a degree $12$ totally ramified extension of $F$,
then $\bG$ is unramified over $F^{\mathrm{t}}$ and (iii) is obtained.

For assertion (iv), let $F^{\mathrm{t}}$ be a field in (iii) and
$x^{\mathrm{t}}$ the centroid of a facet of $\mathscr{B}(\bG_{F^{\mathrm{t}}},F^{\mathrm{t}})$
whose stabilizer is  a maximal compact subgroup of $\bG(F^{\mathrm{t}})$ containing $\Pi$.
Since $\bG_{F^{\mathrm{t}}}$ is semisimple and unramified by (iii), there exists
a totally ramified extension $F'$ of $F^{\mathrm{t}}$ 
such that $x':=\iota_{F',F^{\mathrm{t}}}(x^{\mathrm{t}})$ 
is a hyperspecial point of $\mathscr{B}(\bG_{F'},F')$ \cite[Lemma 2.4]{La95}.
The stabilizer $\bG(F')^{x'}$ is the desired group $\Omega$.
\end{proof}



\subsection{Commutants and semisimplicity}

Let $F/\Q_\ell$ be a finite field extension, 
$V$ an $n$-dimensional $F$-vector space, 
$\Lattice$ an $O_F$-lattice in $V$,
and $\Gamma$ a closed subgroup of $\GL(\Lattice) \cong \GL_n(O_F)\subset  \GL_n(F)\cong\GL(V)$.  
If $F'$ is a finite extension of $F$, we can regard $\Gamma$ also
as a subgroup of $\GL(\Lattice') \cong \GL_n(O_{F'})$, where $\Lattice' = \Lattice\otimes_{O_{F}} O_{F'}$.  
Let $\pi$ (resp. $\pi'$) be a uniformizer of $O_F$ (resp. $O_{F'}$) and define $V' := V\otimes_F F'$. 
We have the following results in this setting.

\begin{lem}\label{ext}
The group $\Gamma$ acts semisimply
on $V'$ if and only if it acts semisimply on $V$, in which case we have
$$\dim_{F'} (\End_\Gamma V') = \dim_F (\End_\Gamma V).$$
Likewise, $\Gamma$ acts semisimply on the reduction $\ReducedLattice := \Lattice/\pi \Lattice$ if and only if it acts semisimply on 
$\ReducedLattice' := \Lattice'/\pi'\Lattice'$, in which case we have
$$\dim_{O_F/(\pi)} (\End_\Gamma \ReducedLattice) = \dim_{O_{F'}/(\pi')} (\End_\Gamma \ReducedLattice').$$
\end{lem}

\begin{proof}
This is clear.
\end{proof}



\begin{lemma}
\label{coh}
Let $M$ be a free $O_F$-module of
finite rank and $\Gamma$  a subgroup of $\Aut_{O_F}M$.
Then for all $k\ge 1$, the inclusion 
$$M^\Gamma/\pi^k M^\Gamma\subset
(M/\pi^kM)^\Gamma$$
 is either proper for all $k\ge 1$ or is an equality for all $k\ge 1$.

\end{lemma}

\begin{proof}
We use the following diagram of cohomology sequences:
$$\xymatrix{0\ar[r]&M^\Gamma\ar[r]^\pi\ar[d]^1&M^\Gamma\ar[r]\ar[d]^{\pi^{k-1}} &(M/\pi M)^\Gamma\ar[r]\ar[d]^{\pi^{k-1}}& H^1(\Gamma,M)\ar[d]^1 \\
0\ar[r]&M^\Gamma\ar[r]^{\pi^k}&M^\Gamma\ar[r] &(M/\pi^k M)^\Gamma\ar[r]&H^1(\Gamma,M)}.$$
The inclusion follows from the second row.  As the rightmost vertical arrow is an isomorphism, $M^\Gamma/\pi M^\Gamma\subsetneq(M/\pi M)^\Gamma$
implies $M^\Gamma/\pi^k M^\Gamma\subsetneq
(M/\pi^kM)^\Gamma$ for all $k\ge 1$.  Conversely, if $M^\Gamma/\pi M^\Gamma=(M/\pi M)^\Gamma$,
the cohomology sequence
$$0\to (M/\pi M)^\Gamma\to (M/\pi^kM)^\Gamma\to (M/\pi^{k-1} M)^\Gamma\to \cdots$$
implies by induction on $k$ that 
$$|(M/\pi^kM)^\Gamma| \le |(M/\pi M)^\Gamma|^k = |M^\Gamma/\pi M^\Gamma|^k = |M^\Gamma/\pi^k M^\Gamma|$$ 
for all $k\ge 1$, which implies
$(M/\pi^kM)^\Gamma = M^\Gamma/\pi^k M^\Gamma$.
\end{proof}

\begin{lem}\label{ss}
Let  $V$ be a finite-dimensional vector space over a field $F$ and $H\subset G\subset \GL(V)$ be subgroups. 
Let $V^{\ss}$ be the semisimplification of $G$ on $V$. 
The following assertions hold.
\begin{enumerate}[($i$)]
\item $\dim_F(\End_{G} V)\le \dim_F(\End_G V^{\ss})$.
\item If $\dim_F(\End_{G} V)=\dim_F(\End_G V^{\ss})$, then $G$ acts semisimply on $V$.
\item If $H$ acts semisimply on $V$ and $\dim_F(\End_{G} V)=\dim_F(\End_{H} V)$, then $G$ acts semisimply on $V$.
\item If $H$ acts semisimply on $V$ and $\dim_F(\End_{G} V)=\dim_F(\End_{H} V)$, then $H$ is absolutely irreducible on every
absolutely irreducible subrepresentation $W\leq V$ of $G$.
\end{enumerate}
\end{lem}

\begin{proof} 


%
%


Assertions (i) and (ii) are just \cite[Lemma 3.6.1.1]{CHT18}.

Let $H^{\red}$ and $G^{\red}$ be the images of  $H$ and $G$ respectively in $\GL(V^{\ss})$.
Since $H$ acts semisimply on $V$, the representations $H\to \GL(V)$ and $H^{\red}\to \GL(V^{\ss})$ are isomorphic. This implies
\begin{align*}
\dim_F(\End_{G} V)&=\dim_F(\End_{H} V)=\dim_F(\End_{H^{\red}}V^{\ss}) \\
& \geq\dim_F(\End_{G^{\red}}V^{\ss}).
\end{align*}
Then (iii) follows from (i) and (ii).

For assertion (iv), $G$ is semisimple on $V$ by (iii).
The absolute irreducibility of $W$ and the condition $\dim_F(\End_{G} V)=\dim_F(\End_{H} V)$ 
force $1=\dim_F(\End_{G} W)=\dim_F(\End_{H} W)$. We are done since $H$ is semisimple on $W$.
\end{proof}

\begin{prop}
\label{reduce-eq}
Let $F$ be a characteristic $0$ local field with valuation ring $O_F$
and residue field $\F_q$.
Let $V$ be a finite-dimensional vector space over $F$ and
$\Gamma$  a compact subgroup of $\GL(V)$ which acts semisimply on $V$.
The following assertions are equivalent.
\begin{enumerate}[(i)]
\item For some $\Gamma$-stable lattice $\Lambda$ of $V$, we have
$$\dim_{F} (\End_\Gamma V) = \dim_{\F_q} (\End_\Gamma (\Lambda\otimes_{O_F}\F_q)^{\ss}).$$

\item For every $\Gamma$-subrepresentation $W$ of $V$ and every $\Gamma$-stable lattice 
$\Lambda_W$ of $W$, $\Lambda_W\otimes_{O_F}\F_q$ is semisimple, and
$$\dim_{F} (\End_\Gamma W) = \dim_{\F_q} (\End_\Gamma (\Lambda_W\otimes_{O_F}\F_q)).$$
\item The following two assertions hold.
\begin{itemize}
\item[(a)] If $W$ is an irreducible $\Gamma$-subrepresentation of $V$, and $\Lambda_W$ a  $\Gamma$-stable lattice  of $W$, then 
$\Lambda_W\otimes_{O_F}\F_q$ is semisimple and
$$\dim_{F} (\End_\Gamma W )= \dim_{\F_q} (\End_{\Gamma}(\Lambda_W\otimes_{O_F}\F_q)).$$
\item[(b)] If $W_1$ and $W_2$ are non-isomorphic irreducible $\Gamma$-subrepresentations of $V$ and
$\Lambda_1$ and $\Lambda_2$ are $\Gamma$-stable  lattices  of $W_1$ and $W_2$ respectively, then
$\Lambda_1\otimes_{O_F}\F_q$ and $\Lambda_2\otimes_{O_F}\F_q$ have no common irreducible $\Gamma$-subrepresentation.
\end{itemize}
\end{enumerate}
\end{prop}

\begin{proof}
Assume assertion (i), and let $W$ and $W'$ be any subrepresentations of $V$ and $\Lambda$ and $\Lambda'$ stable lattices in $W$ and $W'$ respectively.  Applying 
and Lemma~\ref{coh} to $M = \Hom_{O_F}(\Lambda,\Lambda')\subset \Hom_\Gamma(W,W')$, we obtain

\begin{equation}
\label{ww}
\begin{split}
\dim_{F} (\Hom_\Gamma (W,W')) &= \rk_{O_F} (\Hom_\Gamma(\Lambda,\Lambda')) \\ &\le 
\dim_{\F_q} (\Hom_\Gamma (\Lambda\otimes_{O_F}\F_q,\Lambda'\otimes_{O_F}\F_q)).
\end{split}
\end{equation}
Let $W_1$ and $W_2$ be respectively complementary $\Gamma$-subrepresentations of $V$ 
with $\Gamma$-stable lattices $\Lambda_1$ and $\Lambda_2$.
The Brauer-Nesbitt theorem implies that $(\Lambda\otimes_{O_F}\F_q)^{\ss}$ is the semisimplification
of $(\Lambda_1\oplus\Lambda_2)\otimes_{O_F}\F_q$. It follows by (i) that 
\begin{align*}
\dim_{\F_q} (\End_\Gamma (\Lambda\otimes_{O_F}\F_q)^{\ss}) &\ge 
\dim_{\F_q} (\End_\Gamma \bigl((\Lambda_1\oplus\Lambda_2)\otimes_{O_F}\F_q\bigr)) \\
										&= \sum_{i=1}^2 \sum_{j=1}^2 \dim_{\F_q} (\Hom_\Gamma(\Lambda_i\otimes_{O_F}\F_q,\Lambda_j\otimes_{O_F}\F_q)) \\
										&\ge \sum_{i=1}^2 \sum_{j=1}^2\dim_{F} (\Hom_\Gamma(\Lambda_i\otimes_{O_F}F,\Lambda_j\otimes_{O_F}F) )\\
										&= \dim_{F} (\End_\Gamma V),
\end{align*}
where equality holds only if $(\Lambda_1\oplus\Lambda_2)\otimes_{O_F}\F_q=(\Lambda_1\otimes_{O_F}\F_q)\oplus(\Lambda_2\otimes_{O_F}\F_q)$ is semisimple (by Lemma~\ref{ss}(ii)) 
and equality holds in (\ref{ww}) for $W=W'=W_1$ and $W=W'=W_2$.
This implies (ii).

Assertion (ii) implies (iii-a) trivially and (iii-b) by setting $W= W_1+W_2$.

Given assertion (iii), if $W_1^{a_1}\oplus\cdots\oplus W_k^{a_k}$ is a decomposition of $V$ into pairwise non-isomorphic $\Gamma$-representations, then choosing for each summand $W_i^{a_i}$
a $\Gamma$-stable lattice of the form $\Lambda_i^{a_i}$ and setting $\Lambda = \sum_i \Lambda_i^{a_i}$, we see that $\Lambda\otimes_{O_F}\F_q$ is a direct sum of isotypic semisimple representations
$(\Lambda_i\otimes_{O_F}\F_q)^{a_i}$, where the representations $\Lambda_i\otimes_{O_F}\F_q$ are pairwise without common irreducible factor.
Thus, $\Lambda\otimes_{O_F}\F_q$ is semisimple, and 
\begin{align*}\dim_{F} (\End_\Gamma V) &= \sum_{i=1}^k a_i^2 \dim_{F}(\End_\Gamma W_i )\\
									&= \sum_{i=1}^k a_i^2 \dim_{\F_q}(\End_\Gamma (\Lambda_i \otimes_{O_F}\F_q)) =
\dim_{\F_q} (\End_\Gamma (\Lambda\otimes_{O_F}\F_q)).
\end{align*}
\end{proof}

\begin{cor}
\label{abs-irred-red}
Let $V$ be a finite-dimensional vector space over $\Q_\ell$ and
$\Gamma$  a compact subgroup of $\GL(V)$ which acts semisimply on $V$.
The following assertions are equivalent.
\begin{enumerate}[(i)]
\item For some $\Gamma$-stable lattice $\Lambda$ of $V$, we have
$$\dim_{\Q_\ell} (\End_\Gamma V) = \dim_{\F_\ell} (\End_\Gamma (\Lambda\otimes\F_\ell)^{\ss}).$$
\item If $F$ is a finite extension of $\Q_\ell$ with residue field $\F_q$ such that every irreducible
$\Gamma$-subrepresentation of $V\otimes F$ is absolutely irreducible, then
the following two assertions hold.
\begin{itemize}
\item[(a)] If $W$ is an irreducible $\Gamma$-subrepresentation of $V\otimes F$ 
and $\Lambda_W$ a  $\Gamma$-stable $O_F$-lattice  of $W$, then 
$\Lambda_W\otimes_{O_F}\F_q$ is absolutely irreducible.

\item[(b)] If $W_1$ and $W_2$ are non-isomorphic irreducible $\Gamma$-subrepresentations of $V\otimes F$ and
$\Lambda_1$ and $\Lambda_2$ are $\Gamma$-stable  $O_F$-lattices  of $W_1$ and $W_2$ respectively, then
$\Lambda_1\otimes_{O_F}\F_q$ and $\Lambda_2\otimes\F_q$ are not isomorphic.
\end{itemize}
\end{enumerate}
\end{cor}

\begin{proof}
Let $F$ be the finite extension of $\Q_\ell$ in assertion (ii).
Tensoring by $O_F$ over $\Z_\ell$, we see by Lemma \ref{ext} that assertion (i) is equivalent to assertion (i) of Proposition~\ref{reduce-eq}.
Regarding $\Gamma$ as a subgroup of $\Aut_F(V\otimes F)$, by absolute irreducibility, assertions (iii-a) and (iii-b)
of Proposition~\ref{reduce-eq} correspond to assertions (ii-a) and (ii-b) respectively.
\end{proof}

\begin{lem}
\label{ss-kernel}
Let $F$ be a finite totally ramified extension of $\Q_\ell$ with ring of integers $O_F$ and residue field $\F_\ell$,
$\Lambda$ a finitely generated free $O_F$-module,
and $\Gamma\subset \Aut_{O_F} \Lambda$ a closed subgroup such that the
action of $\Gamma$ on $\Lambda\otimes\F_\ell$ is semisimple.  Then
$$\ker \bigl(\Gamma\to \Aut_{\F_\ell} (\Lambda\otimes\F_\ell)\bigr)$$
is the maximal normal pro-$\ell$ subgroup of $\Gamma$.
\end{lem}

\begin{proof}
The kernel is a closed subgroup of the pro-$\ell$ group 
$$\ker \bigl(\Aut_{O_F} \Lambda\to \Aut_{\F_\ell} (\Lambda\otimes\F_\ell)\bigr)$$ 
and
therefore is again pro-$\ell$.  So it suffices to prove it is maximal among normal pro-$\ell$ subgroup of $\Gamma$.  If not, the image
of any normal pro-$\ell$ subgroup not contained in the kernel is a non-trivial normal $\ell$-subgroup of the image of $\Gamma\to \Aut_{\F_\ell} (\Lambda\otimes\F_\ell)$.
However, a subgroup of $\GL_n(\F_\ell)$ which acts semisimply cannot have a non-trivial normal $\ell$-subgroup, since a semisimple representation of
an $\ell$-group over $\F_\ell$ is necessarily trivial.
\end{proof}

\subsection{Formal characters and regular elements}  \label{formal}
 
\subsubsection{}\label{261} We work over a field $F$ of any characteristic. 
Suppose at first that $F$ is algebraically closed.
Let $\bT\subset \GL_n$ be a torus of rank $r$.  By \emph{weights of $\bT$}, we mean the weights of the ambient representation $\bT\to \GL_n$,
i.e., the characters $\chi\in X^*(\bT)$ appearing in the decomposition of the ambient representation into irreducible factors.
We define $m_\chi$ to be the multiplicity of the weight $\chi$ and $\sum_\chi m_\chi[\chi] \in \Z[X^*(\bT)]$ to be the \emph{formal character
of $\bT$} (as a subgroup of $\GL_n$).  

For $N\in\N$, let $I_N$ denote the set of integers in the interval $[-N,N]$.  
Fixing an isomorphism $i\colon \Z^r\to X^*(\bT)$, the formal character 
is \emph{bounded by $N$ with respect to $i$} if $m_\chi > 0$ only for $\chi\in i(I_N^r)$.  We say it is \emph{bounded by $N$} if this is true for some choice of $i$ (see \cite[Definition 4]{Hu15}).
In this case we say that $\bT$ is an \emph{$N$-bounded} torus.  

For any connected algebraic subgroup $\bG$ of $\GL_{n,F}$, 
we define the \emph{formal character of $\bG$} as the formal character of any maximal torus $\bT\subset \bG$
and we say that $\bG$ is $N$-bounded if $\bT$ is $N$-bounded; since
the maximal tori of $\bG$ are conjugate to one another, this does not depend on the choice of $\bT$.
We say the formal characters of connected algebraic subgroups $\bG_1\subset \GL_{n,F_1}$ and $\bG_2\subset \GL_{n,F_2}$
(where $F_1$ and $F_2$ may even have different characteristics)
are \emph{the same} if there exist maximal tori $\bT_1$, $\bT_2$ (of $\bG_1$, $\bG_2$ respectively)
and an isomorphism $X^*(\bT_1)\to X^*(\bT_2)$ mapping the formal character of $\bT_1\subset\GL_{n,F_1}$ to that of $\bT_2\subset\GL_{n,F_2}$.
This is equivalent to the existence of $g_1\in\GL_n(F_1)$ and $g_2\in\GL_n(F_2)$ such that
$g_1^{-1}\bT_1 g_1\subset\G^n_{n,F_1}$ and  $g_2^{-1}\bT_2 g_2\subset\G^n_{n,F_2}$ are diagonal tori cut out 
by the same set of characters in $\G^n_m$ \cite[Proposition 2.0.1]{Hu15}.
If $F$ is not algebraically closed, the formal character of $\bG\subset\GL_{n,F}$ 
is defined to be the formal character of $\bG_{\overline F}\subset\GL_{n,\overline F}$
and we say the former is $N$-bounded if the latter is $N$-bounded.

\subsubsection{}\label{262}
Let $\bT\subset \G_m^n$ be a rank $r$ diagonal torus over a field $F$.
The corresponding map of character groups
$f\colon \Z^n = X^*(\G_m^n) \to X^*(\bT)$ determines an ordered $n$-tuple $(f(e_1),\ldots,f(e_n))\in X^*(\bT)^n$, where the $e_i$ are the standard generators of $\Z^n$.  The number of occurrences of a character $\chi\in X^*(\bT)$ in this $n$-tuple equals $m_\chi$, so  $\bT$ determines  the element of $X^*(\bT)^n$ up to permutation.
As $\ker f$ is finitely generated, there exists $M\in\N$ such that $I_M^n\cap \ker f$ generates $\ker f$.  
This is equivalent to the fact that $\bT$ is the intersection in $\G_m^n$ 
of $\ker \chi$ over some collection of $\chi\in I_M^n\subset \Z^n = X^*(\G_m^n)$.
There are finitely many homomorphisms $\Z^n\to \Z^r$ sending each $e_i$ to an element of
$I_N^r$, so there
exists a constant $\ConstMN(n,N)$
depending only on $n,N\in\N$ (independent of the field $F$) such that this property holds for all $N$-bounded subtori of $\G^n_m$
whenever $M\geq \ConstMN(n,N)$.

Let $\bT\subset\GL_{n,F}$ be a torus and $t\in\bT(F)$. There exists $g\in\GL_n(\overline F)$
such that $g^{-1}\bT g\subset\G^n_m$ is a diagonal torus. If $h\in\GL_n(\overline F)$ is another element such that 
$h^{-1}\bT h\subset\G^n_m$, then a permutation of coordinates of $\G^n_m$ maps $g^{-1}\bT g$ to $h^{-1}\bT h$.
We say that $t\in \bT(F)$ is \emph{$m$-regular} if whenever
$\chi\in I_m^n\subset X^*(\G_m^n)$ is a character such that $g^{-1}tg\in \ker\chi$, then $g^{-1}\bT g\subset \ker\chi$.  As $I_m^n$ is stable under permutation of coordinates, this does not depend on the choice of $g$ conjugating $\bT$ into a diagonal torus.

If $\bT$ is a maximal torus of a connected reductive subgroup $\bG\subset\GL_{n,F}$ and $t\in \bT(F)$
is $1$-regular, then $t$ is a regular semisimple element of $\bG$; this follows from the fact that the adjoint representation of $\bG$ is  
a subrepresentation of the restriction to $\bG$ of the adjoint representation of $\GL_{n,F}$.
If $t\in\bG(F)$ is regular semisimple, we say that $t$ is $m$-regular if $t$ is $m$-regular with respect to 
the unique maximal torus $\bT\subset\bG$ (defined over $F$) containing $t$.
The following lemmas are fundamental.

\begin{lem}
\label{MN}
Let $F_1,F_2$ be fields and $M\ge \ConstMN(n,N)$ an integer.
If $\bT_1\subset \G^n_{m,F_1}$ and $\bT_2\subset \G^n_{m,F_2}$ 
are $N$-bounded diagonal tori of same rank and
$t\in \bT_1(F_1)$ is an $M$-regular element such that for all $\chi\in I_M^n\cap X^*(\G^n_m)$,
the inclusion $\bT_2\subset \ker\chi$ implies $\chi(t)=1$, 
then $\bT_1$ and $\bT_2$ are cut out by the same set of characters in $I_M^n\cap X^*(\G_m^n)$.
\end{lem}

\begin{proof}
The proof is immediate.
\end{proof}

\begin{lemma}\label{reg-contain}
Let $F$ be a field and $M\ge \ConstMN(n,N)$ an integer.
If $\bT_1$ and $\bT_2$ are $N$-bounded tori of $\GL_{n,F}$ and $t\in \bT_1(F)\cap \bT_2(F)$
is $M$-regular for $\bT_1$, then  $\bT_1\subset \bT_2$.
\end{lemma}

\begin{proof}
Without loss of generality, 
we may assume that $F$ is algebraically closed, $\bT_2\subset \G_m^n$, and $g^{-1}\bT_1 g \subset \G_m^n$ for some $g\in\GL_n(F)$.  
As $g^{-1} t g$ and $t$ are both diagonal, $g$ can be written as $zh$, 
where $z$ commutes with $t$ and $h$ normalizes $\G_m^n$ in $\GL_n$.  Thus we may take $g=z$.
Since $t\in z^{-1}\bT_1 z \subset \G_m^n$ is at least $1$-regular, 
it follows that $\bT_1$ is also diagonal. Since $t$ is $M$-regular,
every character $\chi\in I_M^n\cap X^*(\G_m^n)$ which annihilates $\bT_2$ sends $t$ to $1$ 
and is therefore trivial on $\bT_1$. By the definition of $M$, we obtain  $\bT_1\subset \bT_2$.
\end{proof}

\subsubsection{}\label{263}
Let $A$ be an abelian group and $B$ a subgroup of $A$. The \emph{saturation of} $B$ 
is the subgroup of elements $a\in A$ such that $ma\in B$ for some non-zero $m\in\Z$. If $B$ is 
equal to its saturation, then $B$ is said to be \emph{saturated}.
We focus on algebraic groups over finite fields $\F_\ell$.

\begin{prop}
\label{exp-gen}
If $\ell$ and $N$ are sufficiently large in terms of $n$, every exponentially generated subgroup 
of $\GL_{n,\F_\ell}$ has $N$-bounded formal character.
\end{prop}

\begin{proof}
Let $\bT\subset\G_m^n$ be a torus of rank $r$ over an algebraically closed field $F$
with $f: X^*(\G_m^n)\to X^*(\bT)$.
For $\Sigma\subset \{1,\ldots,n\}$ of cardinality $r$, let $\bT_{\Sigma}\subset \G_m^n$ be the rank $n-r$ torus
with $1$ in the $\sigma$-coordinate for all $\sigma\in\Sigma$ with $f_\Sigma: X^*(\G_m^n)\to X^*(\bT_\Sigma)$. 
 The fiber product over $\G_m^n$ of two subtori is 
cut out by the sum of the subgroups 
of $X^*(\G_m^n)$ cutting out each of the tori.  The closed subscheme $\bC$ of $\G_m^n$ cut out by
a subgroup of $X^*(\G_m^n)$ of index $D$ is reduced and satisfies $|\bC(F)|=D$ 
if and only if $D$ is not divisible by the characteristic of $F$.
If $\bT\times_{\G_m^n} \bT_{\Sigma}$ is reduced, then $|\bT(F)\cap \bT_{\Sigma}(F)|$ and the index
$$[\Lambda^n X^*(\G_m^n) : \Lambda^r (\ker f)\wedge \Lambda^{n-r} (\ker f_\Sigma)]$$
have the same cardinality.  
Thus, if $B$ is fixed, $\bT\times_{\G_m^n} \bT_{\Sigma}$ is reduced for all $\Sigma$,
and $|\bT(F)\cap \bT_{\Sigma}(F)|\le B$ for all $\Sigma$ for which the intersection is finite, then
there are only a finite number of possibilities for the top exterior power $\Lambda^r (\ker f)\subset \Lambda^r X^*(\G_m^n)$ and therefore a finite number of possibilities for $(\ker f)\otimes\Q$ as a subspace of $X^*(\G_m^n)\otimes \Q = \Q^n$.  As $\ker f$ is saturated, 
$(\ker f)\otimes\Q$ determines $\ker f$ as a subgroup of $X^*(\G_m^n)$ and therefore determines the formal character of $\bT$.

By \cite[Proposition~3]{La10} and the fact that Hilbert schemes are of finite type, 
 the exponentially generated subgroups of  $\GL_{n,\F_\ell}$ for all sufficiently large $\ell$ form a constructible family 
$\cG\subset \GL_{n,\cS}$ in the sense of \cite{LP}; i.e., 
$\cG$ is a closed subscheme of the general linear scheme over a scheme $\cS$ of finite type over $\Spec \Z$ 
and for every algebraically closed field $F$ of characteristic $0$ or sufficiently large positive characteristic, 
every exponentially generated subgroup
of $\GL_{n,F}$ is of the form $\cG_x$ for some $x\in \cS(F)$.

Let $\cT$ denote the closed subscheme of $\GL_{n,\cS}$ consisting of diagonal matrices, and 
for any subset $\Sigma\subset \{1,2,\ldots,n\}$, let $\cT_\Sigma$ denote the closed subscheme of $\cT$ for which the $\sigma$-coordinate is $1$ for all $\sigma\in \Sigma$.
The fiber product $\cG\times_{\GL_{n,\cS}} \cT_\Sigma$
is a group scheme over $\cS$, therefore reduced over every point in characteristic $0$ and therefore
reduced over every point in sufficiently large finite characteristic \cite[Th\'eor\`eme~9.7.7(iii)]{EGAIV}.
Moreover, by \cite[Corollaire~9.7.9]{EGAIV}, there is an upper bound for the cardinality of any finite fiber, and this implies there are only finitely many possibilities for the formal character of any fiber of 
$\cG$.
\end{proof}

\begin{prop}
\label{old3-2}
There exists a constant $\ConstSubtorus(r,k,N)$ depending only on $r,k,N\in\N$ such that if 
$\uT$ is a rank $r$ torus over $\F_\ell$ with $\ell > \ConstSubtorus(r,k,N)$ and 
the $\Gal_{\F_\ell}$-orbit of $\chi \in X^*(\uT_{\overline\F_\ell})$ is $N$-bounded with respect to 
some isomorphism $i\colon \Z^r\to X^*(\uT_{\overline\F_\ell})$, then
\begin{equation}
\label{subtorus-ineq}
|\{t\in \uT(\F_\ell)\mid \chi(t)=1\}| < k^{-1}|\uT(\F_\ell)|.
\end{equation}
\end{prop}

\begin{proof}
Let $X$ be the subgroup of $X^*(\uT_{\overline\F_\ell})$ generated by the Galois-orbit $O_\chi$ of $\chi$.
Then the number of possibilities for $i^{-1}(X)$, 
is bounded by a constant depending only on $N$ and $r$.
Therefore, there exists $s\in\N$, depending only on $N$ and $r$,
such that for all $\xi$ in the saturation of $X$, we have $s\xi\in X$.
It follows that if $t\in \cap_{\xi\in X} \ker\xi$, then $t^s$ belongs to the subtorus $\bT_{O_\chi}$ of $\bT_{\overline\F_\ell}$ cut out by the saturation of $X$.  Since any element of $\uT(\F_\ell)$ in 
$\ker\chi$ is in $\ker\chi^\sigma$ for all $\chi^\sigma\in O_\chi$, it follows that
$\{t\in \uT(\F_\ell)\mid \chi(t)=1\}$ lies in the union of at most $s^r$ translates of a
proper $\F_\ell$-subtorus of $\uT$.  A proper subtorus has at most $(\ell+1)^{r-1}$ $\F_\ell$-points, while $|\uT(\F_\ell)| \ge (\ell-1)^r$, and the proposition follows.
\end{proof}

\begin{cor}
\label{M-regular}
There exists a constant $\ConstReg(\epsilon,m,n,N)$ 
depending only on $\epsilon > 0$ and $m,n,N\in\N$
such that if $\uT\subset \GL_{n,\F_\ell}$ is an $N$-bounded torus with $\ell > \ConstReg(\epsilon,m,n,N)$, 
then the number of elements of $\uT(\F_\ell)$ which fail to be $m$-regular is less than $\epsilon |\uT(\F_\ell)|$.
\end{cor}

\begin{proof}
Taking 
$$\ConstReg(\epsilon,m,n,N) = \max_{0\le r\le n}\ConstSubtorus(r,\lceil |I_m^n|/\epsilon \rceil,mN),$$
this follows immediately.
\end{proof}

\begin{prop}
\label{big-inter}
There exists a constant $\ConstInter(\epsilon,n,N)$ depending only on 
$\epsilon>0$ and $n,N\in\N$ such that if 
$\uG_1$ and $\uG_2$ are $N$-bounded connected subgroups of $\GL_{n,\F_\ell}$ with
$\uG_1$ reductive, $\ell > \ConstInter(\epsilon,n,N)$, and
$$|\uG_1(\F_\ell)\cap \uG_2(\F_\ell)| > \epsilon |\uG_1(\F_\ell)|,$$
then $\uG_1\subset \uG_2$.
\end{prop}

\begin{proof}
Let $r$ be the rank of $\uG_1$ and $M\ge \ConstMN(n,N)$ an integer.
By Corollary~\ref{M-regular}, every maximal torus $\uT$ of $\uG_1$ defined over $\F_\ell$
contains $o((\ell+1)^r)$ elements which fail to be $M$-regular.  Each regular semisimple element belongs to a unique $\uT$,
so the number of maximal tori containing a regular semisimple element defined over $\F_\ell$
is $O(\ell^{-r}|\uG_1(\F_\ell)|)$.  Thus the number of regular semisimple elements of $\uG_1(\F_\ell)$ which
are not $M$-regular is $o(|\uG_1(\F_\ell)|)$.
By Lang-Weil and the root datum of the connected reductive $\uG_{1,\overline\F_\ell}$ has finitely many possibilities (depending on $n$), 
the number of elements of $\uG_1(\F_\ell)$ which are not regular semisimple is also $o(|\uG_1(\F_\ell)|)$.
We conclude that  if $\ell$ is sufficiently large, more than $(1-\epsilon/2)|\uG_1(\F_\ell)|$ elements $x$ of $\uG_1(\F_\ell)$ are 
regular semisimple and are $M$-regular and therefore
do not lie in $\uG_2(\F_\ell)$
unless the unique maximal torus of $\uG_1$ containing $x$ is contained in some maximal torus of $\uG_2$ by Lemma \ref{reg-contain}.

It follows that $\uG_1$ and $\uG_2$ have at least $(\epsilon/3)  |\uG_1(\F_\ell)|$  elements 
 in common which are regular semisimple and $M$-regular for $\uG_1$, if $\ell$ is sufficiently large.
The group generated by the unique maximal  tori of $\uG_1$ containing these elements is a closed connected subgroup of $\uG_1\cap\uG_2$ containing at least $(\epsilon/3) |\uG_1(\F_\ell)|$ regular semisimple elements.
However, if it is a proper subgroup of $\uG_1$, its dimension is at most $\dim \uG_1 - 1$, so if $\ell$ is sufficiently large, it contains less than $(2/\ell)  |\uG_1(\F_\ell)|$
elements.  Thus $\uG_1\subset \uG_2$.
\end{proof}

\section{Maximality of compact subgroups}

\subsection{Theorem~\ref{gpmain}}
\label{main-thm}

The main goal of this subsection is to establish the following theorem.

\begin{customthm}{\ref{gpmain}}
Let $\bG\subset \GL_{n,\Q_\ell}$ be a connected reductive subgroup,
$\uG\subset \GL_{n,\F_\ell}$ a connected reductive subgroup with $\uG^{\der}$ as derived group and $\uZ$ as connected center, 
$\Gamma$ a closed subgroup of $\bG(\Q_\ell)\cap \GL_n(\Z_\ell)$, and
$\rd\colon \Gamma\to \GL_n(\F_\ell)$ a semisimple continuous representation with $G:=\phi(\Gamma) \subset \uG(\F_\ell)$.
Assume that this data satisfies the following conditions.
\begin{enumerate}[(a)]
\item The subgroup $\Gamma$ is Zariski-dense in $\bG$.
\item There is an equality of semisimple ranks: $\rank \bG^{\der} = \rank \uG^{\der}$.
\item The derived group $\uG^{\der}$ is exponentially generated.
\item For all $\gamma\in \Gamma$, the (mod $\ell$) reduction of the characteristic polynomial of $\gamma$ is the characteristic polynomial of $\rd(\gamma)$.
\item The index $[\uG(\F_\ell):G]$ is bounded by $k\in\N$.
\item The formal character of $(\finiteAlgebraicZ,\F_\ell^n)$ is bounded by $N\in\N$, 
where $\finiteAlgebraicZ$ is the connected center of $\uG$.
\item Condition ($\ast$) holds for $\Gamma$ and $G$, i.e.,
$$\dim_{\Q_\ell} (\End_{\Gamma} (\Q_\ell^n)) = \dim_{\F_\ell} (\End_{G} (\F_\ell^n)).$$

\end{enumerate}
If $\ell$ is sufficiently large in terms of the data in (a)--(g), then
the reduction representation $\Gamma\hookrightarrow\GL_n(\Z_\ell)\to\GL_n(\F_\ell)$ and $\phi$ are 
conjugate, $\Gamma^{\sc}$ is a hyperspecial maximal compact subgroup
of $\bG^{\sc}(\Q_\ell)$, and $\bG^{\der}$ is unramified.
Hypotheses (a)--(f) of Theorem~\ref{gpmain}  suffice to imply that 
$\bG^{\der}$ splits over some finite unramified extension of $\Q_\ell$ and is unramified over 
every degree $12$ totally ramified extension of $\Q_\ell$.
\end{customthm}

\subsection{The condition $(\ast)$}\label{s3.2}
Suppose the conditions (a)--(f) of Theorem \ref{gpmain} hold.
The goal of this subsection is reduce the condition
$(\ast)$ of Theorem \ref{gpmain}(g) to the \emph{semisimple part}, that is, 
the condition \eqref{reducetoder}
in Proposition \ref{Endss}.

\begin{prop}
\label{commu-var}
There exists a constant $\ConstH(k,n,N)$ depending only on $k,n,N\in\N$ such that if
$\uG\subset \GL_{n,\F_\ell}$ is an $N$-bounded connected reductive subgroup with $\ell > \ConstH(k,n,N)$
and derived group $\uG^{\der}$ also $N$-bounded,
$G$ a subgroup of index bounded by $k$ in $\uG(\F_\ell)$, and
$\uS$  the Nori group of $G$, then the following assertions hold.
\begin{enumerate}[(i)]
\item The ambient representation $G\to \GL_n(\F_\ell)$ is semisimple.
\item The ambient representation $\uG\to \GL_{n,\F_\ell}$ is semisimple.
\item The derived group of $\uG$ is $\uS$.
\item The commutant of $G$ in $M_n(\F_\ell)$ consists of the $\F_\ell$-points of 
the commutant of $\uG$ in $M_{n,\F_\ell}$.
\end{enumerate}

Parts (i) and (ii) hold without the $N$-bounded assumption.
\end{prop}

\begin{proof}
If $\ell>k$, then $G\cap \uG^{\der}(\F_\ell)$ contains $\uG^{\der}(\F_\ell)[\ell]$ and therefore $\uG^{\der}(\F_\ell)^+$.
If $\ell$ is sufficiently large in terms of $n$, every characteristic $\ell$ representation of
$\uG(\F_\ell)$ is semisimple \cite{Ja}.  The restriction of a semisimple representation to a normal subgroup is always
semisimple, so $\uG^{\der}(\F_\ell)^+$ acts semisimply on $\F_\ell^n$.  
As $\uG^{\der}(\F_\ell)^+$  is normal in $G\cap \uG^{\der}(\F_\ell)$ and of prime-to-$\ell$ index,
by \cite[\S10, Exercise 8]{CR88}, the latter also acts semisimply on $\F_\ell^n$.
On the other hand, $G\cap \uG^{\der}(\F_\ell)$ is the kernel of a homomorphism from $G$ to the group
of $\F_\ell$-points of the torus $\uG/\uG^{\der}$.  It is therefore a normal subgroup of prime-to-$\ell$ index in $G$,
so applying \cite[\S10, Exercise 8]{CR88} again, $G$ acts semisimply on $\F_\ell^n$.

Part (ii) is true for any connected reductive algebraic group if $\ell$ is large compared to $n$.   (See \cite[Theorem~3.5]{La-SS} when $\uS$ is semisimple, \cite{Ja} in general).

For part (iii), we note that the formal character of $\uG^{\der}$ is bounded by hypothesis, while the formal character of $\uS$ is bounded by Proposition~\ref{exp-gen}.
As $\uG^{\der}(\F_\ell)^+ = \uS(\F_\ell)^+$, this group is of bounded index in both $\uG^{\der}(\F_\ell)$ and $\uS(\F_\ell)$.
By Lemma~\ref{eqrank} and Proposition~\ref{DR1}(iii), this implies
$$\dim (\uG^{\der}) = \dim_\ell (\uG^{\der}(\F_\ell)) = \dim_\ell  (\uS(\F_\ell)^+ )= \dim_\ell (\uS(\F_\ell)) =  \dim \uS$$
for sufficiently large $\ell$.
Thus Proposition~\ref{big-inter} gives
$\uG^{\der} =  \uS$.

Let $x\in M_n(\F_\ell)$ commute with $G$ and its centralizer in $\GL_{n,\F_\ell}$ be $\uZ_x$.
Then $G \subset \uG(\F_\ell)\cap \uZ_x(\F_\ell)$.
Now, $\uZ_x$ is the complement in a linear subvariety of $n\times n$-matrices of the
zero-locus of the determinant, so it is irreducible.
Moreover, centralizers form a constructible family, so  their formal characters are $N$-bounded 
(e.g., by the proof of Proposition~\ref{exp-gen}).
Part (iv) follows by applying  Proposition~\ref{big-inter} as $\uG_1=\uG$ and $\uG_2$ ranges over all groups $\uZ_x$.
\end{proof}

\begin{prop}\label{Endss}
Under the hypotheses of Theorem~\ref{gpmain}, there exists a constant $\ConstK(k,n,N)$ depending only 
on $k,n,N\in\N$ such that
if $\ell > \ConstK(k,n,N)$, the following statements hold.

\begin{enumerate}[(i)]
\item For any finite extension $F$ of $\Q_\ell$ with uniformizer $\pi$ and residue field $\F_{\ell^f}$ and 
any $\mathcal{O}_F$-lattice $\Lattice$ of $F^n:=\Q_\ell^n\otimes_{\Q_\ell} F$ fixed by $\Gamma$, 
the reduction representation
$$\Gamma\hookrightarrow\GL(\Lattice)\to\GL(\Lattice/\pi\Lattice)$$
is isomorphic to $\phi\otimes\F_{\ell^f}$ and thus semisimple;

\item The formal characters of $\uG^{\der}$ and $\bG^{\der}$ coincide.

\item The commutator subgroup $G'$ acts semisimply on $\F_\ell^n$ and
\begin{equation*}
\label{reducetoder}
\tag{$\ast'$}
\dim_{\Q_\ell} (\End_{\Gamma'} (\Q_\ell^n)) = \dim_{\F_\ell} (\End_{G'} (\F_\ell^n)),
\end{equation*}
where $\Gamma'$ is the commutator subgroup of $\Gamma$, i.e., the closure of the group generated by commutators.
\end{enumerate}
\end{prop}

\begin{proof}
For assertion (i),  the Brauer-Nesbitt Theorem and \ref{gpmain}(d) imply that 
the semisimplification of $(\Gamma,\Lattice/\pi\Lattice)$ is isomorphic to $\phi\otimes\F_{\ell^f}$.
Lemma \ref{coh} (for $k=1$) and Lemma \ref{ss}(i) produce the inequalities
\begin{equation*}\label{extineq}
\dim (\End_\Gamma(F^n))\leq\dim(\End_\Gamma (\Lattice/\pi\Lattice))\leq \dim(\End_G(\F_{\ell^f}^n))
\end{equation*}
which,  by Lemma \ref{ext} and ($\ast$), are actually equalities. 
Then the $\F_{\ell^f}$-representation $\Lattice/\pi\Lattice$ of $\Gamma$ is semisimple by Lemma \ref{ss}(ii).

For parts (ii) and (iii), we first note that for some $N'$ depending only on $n$ and $N$,
the three groups $\uG^{\der}$, $\uG$, and $\bG^{\der}$ are $N'$-bounded. 
Indeed, $N'$-boundedness of the first is due to \ref{gpmain}(c) and Proposition~\ref{exp-gen};
$N'$-boundedness of second is due to $N'$-boundedness of the first and \ref{gpmain}(f);
and $N'$-boundedness of third follows since in characteristic $0$, by the Weyl dimension formula, 
there are only finitely many possibilities for formal characters of semisimple groups which admit 
a faithful $n$-dimensional representation.

Now for (ii), choose an integer $M\geq \ConstMN(n,N')$ (defined in $\mathsection\ref{262}$).  
Let $\uT$ be a maximal torus of $\uG^{\der}$. 
Then the index $[\uT(\F_\ell):\uT(\F_\ell)\cap G']$ is bounded by a constant depending only on $k$ and $n$.
By Corollary~\ref{M-regular}, if $\ell$ is sufficiently large in terms of
$n$, $N'$, $M$, and $k$, there exists $g\in \uT(\F_\ell)\cap G'$ which is $M$-regular in $\uT$.  
Let $\gamma\in \Gamma'\subset \bG^{\der}(\Q_\ell)$ be any lift of $g$
and $\gamma_{\ss}\in \bG^{\der}(\Q_\ell)$ its semisimple part.  
Let $\bT$ denote a maximal torus of $\bG^{\der}$ which contains $\gamma_{\ss}$.
Let $\bT_{\GL}$ be a maximal torus of $\GL_{n,\Q_\ell}$ containing $\bT$ and $h\in \GL_n(\overline\Q_\ell)$ an element such that
$h^{-1}\bT_{\GL} h$ is diagonal.  Thus,
\begin{equation}\label{element}
h^{-1}\gamma_{\ss} h = \diag(\lambda_1,\ldots,\lambda_n),
\end{equation}
where the $\lambda_i$ are the eigenvalues of $\gamma_{\ss}$.  They are integral over $\Z_\ell$, 
so they reduce to $\bar\lambda_1,\ldots,\bar\lambda_n\in \overline\F_\ell$, the eigenvalues of $g$.
Define 
\begin{equation}\label{torus2}
\bT_2:=h^{-1}\bT h
\end{equation}
the diagonal torus and $\bT_1$ to be some diagonalization of $\uT_{\overline\F_\ell}$
so that $g\in\uT(\F_\ell)$ goes to $\diag(\bar\lambda_1,\ldots,\bar\lambda_n)$.
Since $\bT_1$ and $\bT_2$ have the same rank by \ref{gpmain}(b),
it follows by Lemma~\ref{MN} that they
are cut out by the same set of characters in $I_M^n\cap X^*(\G_m^n)$, which implies (ii).

To prove (iii), we use Corollary~\ref{abs-irred-red} to replace ($\ast$) and (\ref{reducetoder})
by assertions \ref{abs-irred-red}(ii-a) and \ref{abs-irred-red}(ii-b).
We fix a finite extension $F$ of $\Q_\ell$ over which $F^n$ decomposes as a direct sum of absolutely
irreducible representations for $\Gamma$ and $ \Gamma'$.
By Zariski-density, any decomposition of $F^n$ as a direct sum of  irreducible $\bG$-representations gives
a decomposition into irreducible $\Gamma$-representations, and likewise, a decomposition
into $\bG^{\der}$-irreducibles gives a decomposition into $\Gamma'$-irreducibles.
As every $\bG$-irreducible restricts to a $\bG^{\der}$-irreducible, the same is true for $\Gamma$-irreducibles
and $\Gamma'$-irreducibles.  

By hypothesis, $G$ is of bounded index in $\bG(\F_\ell)$.  Thus $G\cap \uG^{\der}(\F_\ell)$ is of bounded
index in $\uG^{\der}(\F_\ell)$, and its inverse image in $\uG^{\sc}(\F_\ell)$ is of bounded index and therefore
equal to $\uG^{\sc}(\F_\ell)$ if $\ell$ is sufficiently large.  Thus, $G'$ contains the image of $\uG^{\sc}(\F_\ell)\to \uG^{\der}(\F_\ell)$,
which is of bounded index in $\uG^{\der}(\F_\ell)$.  
Applying Proposition~\ref{commu-var} and Lemma \ref{ss}(iv) 
to $G\subset \uG(\F_\ell)$ and $G'\subset \uG^{\der}(\F_\ell)$, we conclude that
an $\overline\F_\ell$-subspace of $\overline\F_\ell^n$ is invariant and irreducible for $G$, if and only if
it is so for $\uG(\overline\F_\ell)$, if and only if it is so for $\uG^{\der}(\overline\F_\ell)$, if and only if it is so for $G'$.
Hence, we obtain \ref{abs-irred-red}(ii-a) for $\Gamma'$.

For \ref{abs-irred-red}(ii-b), it suffices to show that if $W_1\ncong W_2$ are
irreducible subrepresentations of $\Gamma'$ (equivalently $\bG^{\der}$) in $\overline\Q_\ell^n$, then 
their reductions as irreducible representations of $G'$ 
are non-isomorphic for $\ell$ larger than some constant depending only on $k,n,N$. 
Since $\diag(\bar\lambda_1,\ldots,\bar\lambda_n)$ (the reduction of \eqref{element})
and the diagonal torus in \eqref{torus2} 
are annihilated by the same set of characters in $I_M^n$, the actions of $\gamma_{\ss}$ 
on the reductions of $W_1$ and $W_2$ are isomorphic if and only if the actions of $\bT$ (a maximal torus of $\bG^{\der}$)
on $W_1$ and $W_2$ are isomorphic. We are done.
\end{proof}

\begin{prop}
\label{primetonon}
Let $\Gamma\subset \GL_n(\Q_\ell)$ be a compact subgroup, $\Lambda\subset \Q_\ell^n$ a $\Gamma$-stable lattice,  
$\Delta$ a closed normal subgroup of $\Gamma$, and $\gamma\in \Gamma$.  Assume the following conditions hold:
\begin{enumerate}
\item[(a)] $\gamma$ is a semisimple element of $\GL_n(\Q_\ell)$;
\item[(b)] every element in $M_n(\Q_\ell)$ which commutes with $\Delta$ and with $\gamma$ commutes with $\Gamma$;
\item[(c)] if $\lambda_1,\lambda_2\in \overline\Q_\ell$ are distinct eigenvalues of $\gamma$ then $\lambda_1-\lambda_2$ is an $\ell$-adic unit.
\end{enumerate}
Then
\begin{equation}
\label{normal-star}
\dim_{\Q_\ell}(\End_\Delta(\Q_\ell^n)) = \dim_{\F_\ell}(\End_\Delta(\Lambda\otimes\F_\ell)^{\ss})
\end{equation}
implies
\begin{equation}
\label{general-star}
\dim_{\Q_\ell}(\End_\Gamma(\Q_\ell^n)) = \dim_{\F_\ell}(\End_\Gamma(\Lambda\otimes\F_\ell)^{\ss}).
\end{equation}
\end{prop}

\begin{proof}
The left hand side of (\ref{general-star}) is the dimension of the centralizer of $\Gamma$ in $M_n(\Q_\ell)$,
which by (b) is the dimension of the centralizer of $\gamma$ in $\End_\Delta(\Q_\ell^n)$.
By (a), this is the dimension of the $1$-eigenspace of $\gamma$ acting on
$\End_\Delta(\Q_\ell^n)\subset M_n(\Q_\ell)$ by conjugation.

Defining $M := \End_\Delta \Lambda$, we have $M\otimes_{\Z_\ell}\Q_\ell = \End_\Delta(\Q_\ell^n)$.
Equation \eqref{normal-star} and Lemma \ref{ss}(ii) imply that $\Delta$ is semisimple on $\Lambda\otimes\F_\ell$.
Conditions (a) and (c) imply that $\gamma$ is semisimple on $M\otimes\F_\ell$.
Hence, the right hand side  of (\ref{general-star}) is bounded above by the $\F_\ell$-dimension of the $1$-eigenspace 
of $\gamma$ acting on $M\otimes\F_\ell$.  
By (c) and the first paragraph, this is equal to the left hand side of \eqref{general-star}, which implies the two are equal.
\end{proof}

\subsection{Proof of Theorem \ref{gpmain}}

\begin{proof}
We assume that $\ell>k$, which means that every element of $\uG(\F_\ell)$ of order $\ell$ lies in $G$.
As $\uG(\F_\ell)/\uG^{\der}(\F_\ell)$ has prime to $\ell$ order, 
$$G[\ell] = \uG(\F_\ell)[\ell] = \uG^{\der}(\F_\ell)[\ell],$$
so by \ref{gpmain}(c) and Theorem~\ref{NTB}(iii), the Nori group of $G$ equals $\uG^{\der}$.
As $G$ acts semisimply on $\F_\ell^n$, its maximal normal $\ell$-subgroup is trivial.
The composition $\Gamma\hookrightarrow\GL_n(\Z_\ell)\to\GL_n(\F_\ell)$ is a semisimple representation by Corollary~\ref{abs-irred-red},
so by \ref{gpmain}(d) and Brauer-Nesbitt, it is conjugate to $\phi$.

We now suppose the theorem known in the case that $\bG$ and $\uG$ are semisimple.
We defined
$\Gamma'$ to be the topological group generated by commutators in $\Gamma$, but in fact every element of $\Gamma'$ is a finite product of commutators.  Indeed,
the commutator morphism $\bG\times \bG\to \bG^{\der}$ factors through $\bG^{\ss}\times \bG^{\ss}$.  Now $\Gamma^{\ss}$ is a compact
Zariski-dense subgroup of the $\Q_\ell$-points of a semisimple algebraic group, so it is open, by a theorem of Chevalley.
As the generalized commutator morphism $\bG^{\ss}\times \bG^{\ss}\to \bG^{\der}$
is dominant,  the implicit function theorem implies that the set of commutators of elements of $\Gamma^{\ss}$ in $\bG^{\der}(\Q_\ell)$ has non-empty interior.
It follows that every element in $\Gamma'$ can be written as a finite product of commutators.

If $\phi'\colon \Gamma'\to G$ denotes the restriction of $\phi$ to $\Gamma'$, it follows that
$\phi'(\Gamma') = G'$.
Note that $\phi'$ is semisimple, since $G'$ is a normal subgroup of $G$, and the restriction of a semisimple representation to a normal subgroup is again semisimple.
Conditions \ref{gpmain}(a)--(d) for $(\bG^{\der},\uG^{\der},\Gamma', G',\phi')$
are immediate from the same conditions for $(\bG,\uG,\Gamma,G,\phi)$, while \ref{gpmain}(f) is trivial.
By Proposition~\ref{Endss}, \ref{gpmain}(g) for $\Gamma'$ and $G'$ follows from \ref{gpmain}(g) for $\Gamma$ and $G$.

For \ref{gpmain}(e), we note that $G\cap \uG^{\der}(\F_\ell)$ is of index $\le k$ in $\uG^{\der}(\F_\ell)$.
Assuming $\ell > k$, $G$ contains all elements of $\uG^{\der}(\F_\ell)[\ell]$, so
$G\supset \uG^{\der}(\F_\ell)^+$.
The index of $ \uG^{\der}(\F_\ell)^+$ in $\uG^{\der}(\F_\ell)$ is bounded by $2^{n-1}$ by Theorem~\ref{NTB}(ii).
Therefore, at the cost of replacing the index $k$ by $2^{n-1}$, we may assume that all conditions \ref{gpmain}(a)--(g)
hold for $(\bG^{\der},\uG^{\der},\Gamma', G',\phi')$, while $\bG^{\der}$ and $\uG^{\der}$ are semisimple.

Applying the theorem in the semisimple case, we conclude that the inverse image of $\Gamma'$ in $\bG^{\sc}(\Q_\ell)$ is a hyperspecial maximal compact subgroup.
The central isogeny $\bG^{\der}\to \bG^{\ss}$ maps $\Gamma'$ to $(\Gamma^{\ss})'$, so the inverse image of $\Gamma^{\ss}$ in $\bG^{\sc}(\Q_\ell)$, which is compact, contains the inverse
image of $\Gamma'$, which is maximal compact.  This implies the theorem
in the reductive case.

Thus, we may assume without loss of generality that $\bG$ and $\uG$ are semisimple.
By definition, $\Gamma^{\sc}$ is the inverse image of $\Gamma$ in $\bG^{\sc}(\Q_\ell)$.
By Corollary~\ref{index}, the isogeny $\bG^{\sc}(\Q_\ell)\to \bG(\Q_\ell)$ has bounded cokernel,
so we may replace $\Gamma$ with the image of $\Gamma^{\sc}\to \Gamma$
and $G$ with the new $\phi(\Gamma)$ at the cost of increasing $k$ by a bounded factor.
This will not affect \ref{gpmain}(g); the left hand side of ($\ast$) is unchanged since $\Gamma$ is Zariski dense in 
$\bG$, and $\bG$ is connected, while the left hand side is unchanged by Proposition~\ref{commu-var}(iv).

Thus, we may assume $\Gamma^{\sc}$ maps onto $\Gamma$, and for $\Gamma$-representations,
$\Gamma^{\sc}$-invariance is the same as $\Gamma$-invariance.
By Lemma~\ref{ss-kernel}, 
$G$ is the quotient of $\Gamma$ by its maximal normal pro-$\ell$ subgroup.
Let 
\begin{equation}\label{maxcpt}
\Pi\subset\bG(\Q_\ell)
\end{equation}
be a maximal compact subgroup containing $\Gamma$.
Then $\Pi^{\sc}\subset\bG^{\sc}(\Q_\ell)$ is a maximal compact subgroup containing $\Gamma^{\sc}$ 
and fixes some vertex $x_0$ in the Bruhat-Tits building $\mathscr{B}(\bG^{\sc},\Q_\ell)$. 
The vertex $x_0$ corresponds to a group scheme  $\GroupSchemeZ/\Z_\ell$
with generic fiber isomorphic to $\bG^{\sc}$.
When $\ell$ is large enough depending on $k$, the total $\ell$-rank of $\Gamma$ 
is equal to $\rank\bG^{\sc}$ by \ref{gpmain}(b) and \ref{gpmain}(e).
Lemma \ref{subgroup} and Theorem \ref{compare}(i) imply
that the total $\ell$-rank of  $\Pi$ (and hence $\Pi^{\sc}=\GroupSchemeZ(\Z_\ell)$) is equal to $\rank\bG^{\sc}$ if $\ell$ is large enough depending on $n$.
Then Theorem~\ref{compare}(iv) implies that 
there exist some finite totally ramified extension $F/\Q_\ell$ and a hyperspecial maximal compact subgroup $\Omega$
of $\bG^{\sc}(F)$ corresponding to a semisimple 
group scheme $\GroupSchemeO/O_F$ ($\mathsection\ref{241}$) such that 
$$\Gamma^{\sc}\subset \GroupSchemeZ(\Z_\ell)\subset \Omega=\GroupSchemeO(O_F)\subset\bG^{\sc}(F).$$


As $\GroupSchemeO(O_F)$ is compact, it stabilizes some $O_F$-lattice $\Lambda\subset F^n$.
Let $\psi\colon \Gamma^{\sc}\to \GL(\Lambda\otimes\F_\ell)\cong \GL_n(\F_\ell)$ denote the composition of the maps
$$\Gamma^{\sc}\hookrightarrow \GroupSchemeZ(\Z_\ell)\hookrightarrow \GroupSchemeO(O_F)\to \GL_{O_F} \Lambda
\to \GL_{\F_\ell} (\Lambda\otimes \F_\ell).$$
By Brauer-Nesbitt, $\psi^{\ss}$ is equivalent to $\phi$, so by 
Corollary~\ref{abs-irred-red}, $\psi$ is equivalent to $\phi$.
As $\Gamma^{\sc}$ and $\GroupSchemeO(O_F)$ are both Zariski-dense in $\bG^{\sc}_F$,
\begin{equation}
\label{ineq1}
\begin{split}
\dim_F (\End_{\Gamma^{\sc}} (\Lambda\otimes_{O_F}F)) &=  \dim_F (\End_{\GroupSchemeO(O_F)} (\Lambda\otimes F)) \\
&\le \dim_{\F_\ell} (\End_{\GroupSchemeO(O_F)} (\Lambda\otimes\F_\ell)).
\end{split}
\end{equation}
By \ref{gpmain}(g) and $\Gamma^{\sc}\subset \GroupSchemeO(O_F)$,
\begin{equation}
\label{ineq2}
\begin{split}
\dim_F (\End_{\Gamma^{\sc}} (\Lambda\otimes_{O_F}F)) &= \dim_{\F_\ell}(\End_{\Gamma^{\sc}} (\Lambda\otimes\F_\ell)) \\
& \ge \dim_{\F_\ell} (\End_{\GroupSchemeO(O_F)} (\Lambda\otimes\F_\ell)).
\end{split}
\end{equation}
It follows that equality holds in both (\ref{ineq1}) and (\ref{ineq2}).
By Lemma~\ref{ss}(iii), 
$\GroupSchemeO(O_F)$  acts semisimply on $\Lambda\otimes\F_\ell \cong\F_\ell^n$.

By Lemma~\ref{ss-kernel}, the image $I$ of $\GroupSchemeO(O_F)$ in $\GL (\Lambda\otimes \F_\ell)$ is the quotient of
$\GroupSchemeO(O_F)$ by its maximal normal pro-$\ell$ subgroup, which is the quotient of $\GroupSchemeO(\F_\ell)$ by
a subgroup $Z$ of its center. 
So $I$ is a subgroup of bounded index of the $\F_\ell$-points of the semisimple group 
$\GroupSchemeO_{\F_\ell}/Z$ (isogenous to $\GroupSchemeO_{\F_\ell}$).
As the image of $\psi$ is contained in the image of $\GroupSchemeO(O_F)$
in $\GL(\Lambda\otimes \F_\ell)$, we obtain an embedding of $G$ in $I$, with
$$\dim_{\F_\ell} (\End_{G}(\Lambda\otimes\F_\ell)) = \dim_{\F_\ell} (\End_{I}(\Lambda\otimes\F_\ell)) = 
\dim_{\Q_\ell} (\End_{\Gamma^{\sc}} (\Q_\ell^n)).$$
The image $H$ of $\GroupSchemeZ(\Z_\ell)$ in $I$ satisfies $G\subset H\subset I$.

Let $\uH$ and $\uI$ denote the Nori groups of $H$ and $I$ respectively, so $\uG\subset\uH\subset \uI$.  If $\ell$ is sufficiently large, $\uI$ is  semisimple by Proposition~\ref{serre}(i) and
of rank equal to $\rk_\ell I=\rank(\GroupSchemeO_F/Z)=\rank\GroupSchemeO_F$ by Lemma~\ref{eqrank} and Propositions \ref{DR0} and \ref{DR1}. By Proposition~\ref{commu-var}(iv), the commutants of $\uG$ and $\uI$ in $\End (\Lambda\otimes\F_\ell)$ have the same dimension; they must therefore be the same.   By hypothesis, $\uG$ is  semisimple, and we have equality of ranks:
$$\rank \uG = \rank \bG = \rank \GroupSchemeO_F =  \rank \uI.$$
By the Borel--de Siebenthal Theorem \cite{Gi}, $\uG = \uI$, and it follows that $\uH=\uG$ is likewise  semisimple.  We have
$$\uG(\F_\ell)^+\subset G\subset H\subset I\subset \uG(\F_\ell).$$

As $H$ acts semisimply, since the image of $\ker(\cH(\Z_\ell)\to \cH(\F_\ell))$ in $H$ is a normal $\ell$-subgroup, it must be trivial.  Thus, $H$ is a quotient of $\cH(\F_\ell)$.  
If the vertex $x_0\in \mathscr{B}(\bG^{\sc},\Q_\ell)$ associated to $\cH$ is not hyperspecial, then the unipotent radical of $\cH_{\F_\ell}$ is non-trivial.  
Since $\cH$ is flat, the dimension of $\cH_{\F_\ell}$ equals the dimension of $\bG$, which is also the dimension of $\GroupSchemeO_{\F_\ell}$
and therefore the dimension of $\uG=\uH=\uI$.  
By Proposition \ref{DR1}(iii), we obtain 
$$\dim_\ell H\le \dim_\ell (\cH(\F_\ell)) = \dim \cH_{\F_\ell}^{\ss} < \dim \cH_{\F_\ell} = \dim \uG = \dim_\ell G,$$
which is impossible by  Lemma \ref{subgroup} since $G\subset H\subset\GL_n(\F_\ell)$.
Thus $\cH(\Z_\ell)$ is a hyperspecial maximal compact subgroup,
which means that $\cH(\F_\ell)$ is the group of $\F_\ell$-points of a simply connected semisimple algebraic group over $\F_\ell$,
and $H$ is a quotient of $\cH(\F_\ell)$ by a subgroup of its center.  
As $G$ is of bounded index in $H$, the compact subgroup $\Gamma^{\sc}$ of $\cH(\Z_\ell)$ maps onto
a bounded index subgroup of $\cH(\F_\ell)$ which has to be $\cH(\F_\ell)$ itself when $\ell$ is sufficiently large.   
By a theorem of Vasiu \cite[Theorem~1.3]{Va03}, this implies 
$\Gamma^{\sc} = \cH(\Z_\ell)$, as claimed.

Condition ($\ast$) is used only to prove that the commutants of $\uG$ and $\uI$ have
the same dimension.  In any case, we have $\rank \uG = \rank \uI = \rank \bG$ and therefore $\rk_\ell (\uG(\F_\ell)) = \rank \bG$.  As $G$ is of bounded index in $\uG(\F_\ell)$, if $\ell$ is sufficiently large, $\rk_\ell G = \rk_\ell (\uG(\F_\ell))$, 
and since $G$ is a quotient of $\Gamma$, we obtain
$\rank \bG = \rk_\ell \Gamma \leq \rk_\ell \Pi$ by the construction \eqref{maxcpt} and Lemma \ref{subgroup}.  
Theorem~\ref{compare}(i)--(iii) now imply the remaining claims since $\Pi$ is maximal compact in $\bG(\Q_\ell)$.
\end{proof}

\section{Maximality of Galois actions}

\subsection{Algebraic envelopes}\label{41}
Let $\{\rho_\ell\}_\ell$ be the system of $\ell$-adic representations in Theorem \ref{main}.
The monodromy group (resp. algebraic monodromy group) of $\rho_\ell$ is denoted by $\Gamma_\ell$  (resp. $\bG_\ell$).
The quotient of $\bG_\ell$ by its unipotent radical is denoted by $\bG_\ell^{\red}$.
If $X$ is a projective non-singular variety over $K$, then 
for each $\ell$, the image of $H^i(X_{\overline K},\Z_\ell)$ in $H^i(X_{\overline K},\Q_\ell)\cong \Q_\ell^n$
is a $\Z_\ell$-lattice $\Lambda_\ell$ stabilized by $\rho_\ell$, and
\begin{equation}\label{modl}
\finiteRhoSS:\Gal_K\to \GL_n(\F_\ell)
\end{equation}
denotes the semisimplification of the (mod $\ell$) reduction of $\rho_\ell:\Gal_K\to \GL_n(\Z_\ell)$ 
(the action of $\Gal_K$ on this lattice).
Denote by $\finiteG_\ell$ the image $\finiteRhoSS(\Gal_K)$ for all $\ell$. 
In \cite{Hu15}, we construct the algebraic envelope
$\finiteAlgebraicG_\ell$
(a connected reductive subgroup of $\GL_{n,\F_\ell}$) 
of $\finiteG_\ell$ 
to study the $\ell$-independence of the total $\ell$-rank and the $\mathfrak{g}$-type $\ell$-rank of $\finiteG_\ell$ 
for all sufficiently large $\ell$.
The idea of constructing such a $\finiteAlgebraicG_\ell$ is due to Serre \cite{Se86a},
who considered the Galois action on the $\ell$-torsion points of abelian varieties without complex multiplication (see also \cite{Ca15}).
The algebraic envelope $\finiteAlgebraicG_\ell$ can be written as $\finiteAlgebraicS_\ell\finiteAlgebraicZ_\ell$ where
$\finiteAlgebraicS_\ell:=\finiteAlgebraicG_\ell^{\der}$ is also the Nori group  
of $\finiteG_\ell\subset\GL_n(\F_\ell)$ (by \cite[$\mathsection2.5$]{Hu15})
and $\finiteAlgebraicZ_\ell$ is the identity component of  the center of $\finiteAlgebraicG_\ell$ . 
Theorems \ref{ae} and \ref{thmA} below present the key properties of the algebraic envelopes $\finiteAlgebraicG_\ell$.

\begin{thm}\label{ae}\cite[Theorem 2.0.5, the proof of Theorem 2.0.5(iii)]{Hu15}
After replacing $K$ by a finite normal field extension $L$ if necessary, for all sufficiently large $\ell$, 
the algebraic envelope $\finiteAlgebraicG_\ell\subseteq \GL_{n,\F_\ell}$ has the following properties:
\begin{enumerate}
\item[(i)] $\finiteG_\ell$ is a subgroup of $\finiteAlgebraicG_\ell(\F_\ell)$ 
whose index is bounded uniformly independent of $\ell$;
\item[(ii)] $\finiteAlgebraicG_\ell$ acts semisimply on the ambient space;
\item[(iii)] the representations $\{\finiteAlgebraicS_\ell\to \GL_{n,\F_\ell}\}_{\ell\gg0}$ and 
$\{\finiteAlgebraicZ_\ell\to \GL_{n,\F_\ell}\}_{\ell\gg0}$
have bounded formal characters and in particular, $\{\uG_\ell\to \GL_{n,\F_\ell}\}_{\ell\gg0}$
has bounded formal characters.
\end{enumerate}
\end{thm}

\begin{thm}\label{thmA}\cite[Theorem A, Theorem 3.1.1]{Hu15}
Let $\bG_\ell$ be the algebraic monodromy group of $\rho_\ell$.
After replacing $K$ by a finite normal field extension $L$ if necessary, the following statements hold
for all sufficiently large $\ell$.
\begin{enumerate}
\item[(i)] The formal character of $\finiteAlgebraicS_\ell\to \mathrm{GL}_{n,\F_\ell}$ 
(resp. $\uG_\ell\to \mathrm{GL}_{n,\F_\ell}$) is independent of $\ell$ and is equal to the formal character of  
$(\bG_\ell^{\red})^{\der}\to \GL_{n,\Q_\ell}$ (resp. $\bG_\ell^{\red}\to\GL_{n,\Q_\ell}$).
\item[(ii)] The non-abelian composition factors of $\finiteG_\ell$ 
and the non-abelian composition factors of $\finiteAlgebraicS_\ell(\mathbb{F}_\ell)$ are in bijective correspondence.
Thus, the composition factors of $\finiteG_\ell$ are finite simple groups of Lie type in characteristic $\ell$ and cyclic groups.
\end{enumerate}
\end{thm}

Note that the implicit constants in Theorems~\ref{ae} and \ref{thmA} depend only on the system (\ref{modl})
of (mod $\ell$) Galois representations.

\begin{remark}\label{bichar}
The formal bi-character \cite[Definition 2.3]{Hu18} of $\bG_\ell^{\red}\to \GL_{n,\Q_\ell}$ is independent of $\ell$ \cite[Theorem 3.19]{Hu13}.
\end{remark}

\begin{prop}
The system of algebraic envelopes $\{\uG_\ell\subset\GL_{n,\F_\ell}\}_{\ell\gg0}$ 
is characterized by the conditions \ref{ae}(i) and (iii) in the sense that 
if $\{\uH_\ell\subset\GL_{n,\F_\ell}\}_{\ell\gg0}$ is another system of connected reductive subgroups such that
for $\ell\gg0$, $G_\ell$ is a subgroup of $\uH_\ell(\F_\ell)$ whose index is uniformly bounded and 
the formal character of $\uH_\ell$ is uniformly bounded,
then $\uG_\ell=\uH_\ell$ for all sufficiently large $\ell$.
\end{prop}

\begin{proof}
This follows directly from Proposition \ref{big-inter}.
\end{proof}

\begin{thm}\label{connect}
If the algebraic monodromy group $\bG_\ell$ is connected for all $\ell$,
then $G_\ell\subset\uG_\ell(\F_\ell)$ for all sufficiently large $\ell$.
\end{thm}

\begin{proof}
Let $L$ be a finite normal extension of $K$ in Theorem \ref{ae} such that 
$\bar\rho_\ell^{\ss}(\Gal_L)\subset\uG_\ell(\F_\ell)$ for $\ell\gg0$.
Let $\uN_\ell$ be the normalizer of the Nori group $\uS_\ell$ in $\GL_{n,\F_\ell}$.
Then $G_\ell\subset\uN_\ell(\F_\ell)$.

We claim that $G_\ell=\bar\rho_\ell^{\ss}(\Gal_K)$ normalizes the connected reductive group $\uG_\ell$ for $\ell\gg0$.
By construction (see \cite[Proof of Theorem 2.0.5(i) and (ii)]{Hu15}), 
$\uG_\ell$ is the preimage of a torus $\uI_\ell\subset\GL_{W_\ell}$
under some morphism\footnote{The groups $\uI_\ell,\uS_\ell,\uN_\ell,\uG_\ell$ 
are denoted $\bar{\textbf{I}}_\ell,\bar{\textbf{S}}_\ell,\bar{\textbf{N}}_\ell,\bar{\textbf{G}}_\ell$ in \cite{Hu15}.}
$$t_\ell:\uN_\ell\twoheadrightarrow \uN_\ell/\uS_\ell\hookrightarrow\GL_{W_\ell},$$ 
where $W_\ell$ is some $\F_\ell$-vector space whose dimension is bounded independent of $\ell$.
Since the index $[\uI_\ell(\F_\ell):t_\ell(\bar\rho_\ell^{\ss}(\Gal_L))]$ 
and the formal character of $\uI_\ell\subset\GL_{W_\ell}$ are both bounded independent of $\ell$ \cite[Theorem 2.4.2]{Hu15},
the normality of the field extension $L/K$ 
and Proposition \ref{big-inter} imply that $t_\ell(G_\ell)$ normalizes $\uI_\ell$ for all sufficiently large $\ell$.
Hence, the product $G_\ell\uG_\ell$ is a subgroup of $\GL_{n,\F_\ell}$ with identity component $\uG_\ell$ for $\ell\gg0$.

The number of conjugacy classes of elements of the finite group $\Gal(L/K)$ is bounded by $m:=[L:K]$.
Since $\bG_\ell$ is connected for all $\ell$ 
and the strictly compatible system $\{\rho_\ell\}_\ell$ is pure of weight $i$,
the method of Frobenius tori of Serre
(see for example \cite{LP97}, \cite[Theorem 2.6, Corollary 2.7]{Hu18})
implies that there is a Dirichlet density one set of finite places $v$ of $K$ such that 
the \emph{Frobenius torus} $\bT_{\bar v,\ell}\subset\GL_{n,\Q_\ell}$ 
is a maximal torus of $\bG_\ell$ if $v\nmid \ell$ and $\bar v$ 
is some place on $\overline K$ extending $v$ on $K$.
Thus, for each conjugacy class $c$ of $\Gal(L/K)$ we can fix a finite place $v_c$ of $K$ (unramified in $L$) 
mapping to $c$ and a place $\bar v_c$ of $\overline K$ above $v_c$ such 
that $\bT_{\bar v_c,\ell}$ is a maximal torus of $\bG_\ell$ for $\ell\gg0$.
To prove $G_\ell\uG_\ell=\uG_\ell$ for $\ell\gg0$, 
it suffices to show for each $c$ and all sufficiently large $\ell$, the semisimple part 
$\bar\rho_\ell^{\ss}(\text{Fr}_{\bar v_c})_{\ss}\in\uG_\ell(\F_\ell)$. 

For each $c$, there exist a torus $\bT_c\subset\GL_{n,\Q}$ and an element $\gamma_c\in\bT_c(\Q)$ 
such that for all sufficiently large $\ell$, the chain 
\begin{equation}\label{chain1}
\gamma_c\in\bT_c\subset\GL_{n,\Q}
\end{equation}  
is conjugate to the chain
\begin{equation}\label{chain2}
\rho_\ell(\text{Fr}_{\bar v_c})_{\ss}\in \bT_{\bar v_c,\ell}\subset \GL_{n,\Q_\ell}
\end{equation}
by some element in $\GL_n(\Q_\ell)$. For $\ell\gg0$, the reduction modulo $\ell$
\begin{equation}\label{chain3}
g_{c,\ell}\in\uT_{c,\ell}\subset\GL_{n,\F_\ell}
\end{equation} 
 of \eqref{chain1} can be well-defined and 
the two semisimple elements $g_{c,\ell}$ and $\bar\rho_\ell^{\ss}(\text{Fr}_{\bar v_c})_{\ss}$ in $\GL_n(\F_\ell)$
are conjugate since they have the same characteristic polynomial.
Without loss of generality, assume $g_{c,\ell}=\bar\rho_\ell^{\ss}(\text{Fr}_{\bar v_c})_{\ss}$.
The formal characters of $\uT_{c,\ell}$ and $\uG_\ell$ are equal and bounded by some $N\in\N$ independent of $\ell\gg0$ by Theorem \ref{thmA}(i).
Since the powers $\gamma_c^m$ are Zariski dense in $\bT_c$, there exists $M\geq \ConstMN(n,N)$ (in $\mathsection\ref{262}$) 
such that for $\ell\gg0$ the torus $\uT_{c,\ell}\subset\GL_{n,\F_\ell}$ after diagonalization 
is the intersection of the kernels of some characters in $I^n_M\cap X^*(\G^n_m)$ and
the element $g_{c,\ell}^m$ is $M$-regular in $\uT_{c,\ell}$ .
Since we have 
$$g_{c,\ell}^m=\bar\rho_\ell^{\ss}(\text{Fr}_{\bar v_c})_{\ss}^m\in \uG_\ell(\F_\ell),$$
it follows by Lemma \ref{reg-contain} and the $M$-regularity of $g_{c,\ell}^m\in\uT_{c,\ell}$ that $\uT_{c,\ell}\subset \uG_\ell$ for $\ell\gg0$.
We are done since $\bar\rho_\ell^{\ss}(\text{Fr}_{\bar v_c})_{\ss}=g_{c,\ell}\in\uT_{c,\ell}$.
\end{proof}

\subsection{Proof of Theorem \ref{main}}

\begin{proof}
After taking a finite extension $L$ of $K$ and semisimplification of $\rho_\ell$,
we may assume $\bG_\ell$ is connected reductive for all $\ell$ and $G_\ell$ 
is a subgroup of $\finiteAlgebraicG_\ell(\F_\ell)$ for $\ell\gg0$ by Theorem \ref{connect}.
By the constructions of 
\begin{align*}
\begin{split}
\Gamma_\ell\subset\bG_\ell(\Q_\ell)\subset\GL_n(\Q_\ell)\\
G_\ell\subset \finiteAlgebraicG_\ell(\F_\ell)\subset\GL_n(\F_\ell)
\end{split}
\end{align*}
and Theorem \ref{thmA}(i), we are in the setting of Theorem \ref{gpmain},
and the conditions \ref{gpmain}(a)--(d) are verified. Moreover, \ref{gpmain}(e) 
and \ref{gpmain}(f) are verified by Theorem \ref{ae}(i),(iii).
Thus for $\ell\gg0$, the condition ($\ast$) implies hyperspeciality of $\Gamma^{\sc}_\ell$ in $\bG_\ell^{\sc}(\Q_\ell)$
(which in turn implies the unramifiedness of $\bG_\ell^{\sc}$, $\bG_\ell^{\der}$, and $\bG_\ell$ (Proposition \ref{FT}(i))).

Next, we prove the converse. Let $\cG_\ell'$ be the Zariski closure 
of the derived group $\Gamma_\ell'$ in $\GL_{\Lambda_\ell}$ ($\Lambda_\ell$ the lattice in $\Q_\ell^n$) 
endowed with the unique structure of reduced closed
subscheme. For $\ell\gg0$, the group scheme $\cG_\ell'$ is smooth with constant rank 
over $\Z_\ell$ \cite[Theorem 9.1, $\mathsection9.2.1$]{CHT17} and has generic fiber $\bG_\ell^{\der}$.
We first show that $\cG_\ell'$ is semisimple for $\ell\gg0$.
Suppose $\Gamma^{\sc}_\ell$ is a hyperspecial maximal compact subgroup of 
$\bG_\ell^{\sc}(\Q_\ell)$. When $\ell$ is sufficiently large,
there exists a semisimple group scheme $\cH/\Z_\ell$ whose generic fiber is $\bG_\ell^{\sc}$, satisfying $\cH(\Z_\ell)=\Gamma_\ell^{\sc}$.
Since $\Gamma_\ell^{\sc}$ is perfect if $\ell$ is large enough depending on $n$ \cite[Theorem 3.4]{HL15},
it maps into the commutator subgroup $\Gamma_\ell'\subset \cG_\ell'(\Z_\ell)$. Consider
\begin{equation}\label{ahrep}
\bar\rho_\ell:\Gamma_\ell^{\sc}\to\Gamma_\ell'\hookrightarrow\cG'_\ell(\Z_\ell)
\hookrightarrow\GL(\Lambda_\ell)\to\GL(\Lambda_\ell\otimes\F_\ell)
\end{equation}
and let  $\uR_\ell\subset\GL_{\Lambda_\ell\otimes\F_\ell}$ be Nori group of $\bar\rho_\ell(\Gamma_\ell^{\sc})$.
For $\ell$ large enough depending on $n$, the groups
$$\cH(\F_\ell),\ \Gamma_\ell^{\sc},\ \bar\rho_\ell(\Gamma_\ell^{\sc}),\ \uR_\ell(\F_\ell)$$ 
have the same $\ell$-dimension by Theorems \ref{NTB} and \ref{NTC} and the remarks of \S\ref{s2.3.2}.
Then it follows by Proposition \ref{DR1}(iii) that 
\begin{equation}\label{Noriss}
\dim \bG_\ell^{\der}=\dim\cH_{\F_\ell}= \dim_\ell(\cH(\F_\ell))=\dim_\ell(\uR_\ell(\F_\ell))= \dim \uR_\ell^{\ss}.
\end{equation}
Since $\dim \bG_\ell^{\der}\geq \dim\uR_\ell$ by \cite[Theorem 7]{La10} for $\ell$ large enough depending on $n$, 
it follows that the Nori group $\uR_\ell$ is semisimple 
and  the action of $\uR_\ell(\F_\ell)$ (resp. $\uR_\ell(\F_\ell)^+ = \bar\rho_\ell(\Gamma_\ell^{\sc})^+$) on $\Lambda_\ell\otimes\F_\ell$ 
is also semisimple by \cite[Theorem 3.5]{La-SS}.
Since $\bar\rho_\ell(\Gamma_\ell^{\sc})^+$ is normal in $\bar\rho_\ell(\Gamma_\ell^{\sc})$ of prime to $\ell$ index, 
\eqref{ahrep} is semisimple.

Since the equality $\dim\bG_\ell^{\der}=\dim\uR_\ell$ holds, it follows by \cite[Theorem 7(3)]{La10} that 
$\bar\rho_\ell(\Gamma_\ell^{\sc})$ is a subgroup of $\cG_\ell'(\F_\ell)$ 
of index bounded by a constant depending only on $n$.
Hence, if the unipotent radical of the special fiber of $\cG_\ell'$
is non-trivial for some large enough $\ell$, then $\bar\rho_\ell(\Gamma_\ell^{\sc})$ has a non-trivial 
normal subgroup of unipotent elements, which contradicts the semisimplicity of \eqref{ahrep}. 
Thus the group scheme $\cG_\ell'$ is semisimple over $\Z_\ell$.

The weights appearing in the natural $n$-dimensional representation of the generic fiber $\cG'_{\ell,\overline\Q_\ell}$ 
remain bounded as $\ell$ varies.
By a theorem of Springer \cite[Corollary~4.3]{Springer}, if $\ell$ is sufficiently large, the (mod $\ell$) reduction of every irreducible factor in this representation 
is again irreducible, and the (mod $\ell$) reductions of distinct irreducible factors are distinct.
Thus the composition of the special fiber $\cG'_{\ell,\overline\F_\ell}\to \GL_{\Lambda_\ell\otimes\overline\F_\ell}$ 
with the adjoint representation of $\GL_{\Lambda_\ell\otimes\overline\F_\ell}$ is semisimple,
its irreducible factors have bounded highest weights, and they are in one-to-one correspondence with the irreducible factors of the composition
of $\cG'_{\ell,\overline\Q_\ell}\to \GL_{n,\overline\Q_\ell}$ with its adjoint representation. So we obtain
\begin{equation*}
\dim_{\F_\ell} \End_{\cG'_{\ell,\F_\ell}}(\Lambda_\ell\otimes\F_\ell) 
= \dim_{\Q_\ell} \End_{\cG'_{\ell,\Q_\ell}}(\Q_\ell^n) = \dim_{\Q_\ell} \End_{\Gamma_\ell'}(\Q_\ell^n).
\end{equation*}
Since the index of $\Gamma_\ell'$ in $\cG'_\ell(\F_\ell)$ 
and the formal character of $\cG'_{\ell,\F_\ell}\subset\GL_{\Lambda_\ell\otimes\F_\ell}$ are uniformly bounded independent of $\ell$,
it follows by Proposition \ref{commu-var}(i),(iv) that for $\ell\gg0$ the image of $\Gamma_\ell'$ in $\GL(\Lambda_\ell\otimes\F_\ell)$ in \eqref{ahrep} can be identified with the semisimple action $G_\ell'\to\GL(\F_\ell^n)$ and  
$$\dim_{\F_\ell}\End_{G_\ell'}(\F_\ell^n)= \dim_{\F_\ell} \End_{\cG'_{\ell,\F_\ell}}(\Lambda_\ell\otimes\F_\ell)$$
holds. Hence, we deduce \eqref{reducetoder} for $\ell\gg0$.

Now, $\Gamma_\ell/\Gamma'_\ell$ is a Zariski-dense subset of the torus $\bG_\ell/\bG_\ell^{\der}$, which acts on the space of $\Gamma'_\ell$-invariants or, equivalently, the space of $\bG_\ell^{\der}$-invariants, in the adjoint representation of $\GL_{n,\Q_\ell}$.  Thus, any Zariski-dense subset of $\bG_\ell(\Q_\ell)$ contains an element $\gamma$ with the property that any vector in the adjoint representation of $\GL_{n,\Q_\ell}$ which is fixed by $\Gamma'_\ell$ and by $\gamma$ is fixed by $\Gamma_\ell$.

By Theorem~\ref{ae}, $G_\ell$ is of bounded index in $\uG_\ell(\F_\ell)$, so by Proposition~\ref{M-regular}, if $\ell$ is sufficiently
large, there exists a maximal torus $\uT_\ell$ of $\uG_\ell$ and an element $\bar\gamma\in G_\ell\cap \uT_\ell(\F_\ell)$
such that $\bar\gamma$ is not in the kernel of any non-trivial character of $\uT_\ell$ acting in the restriction to $\uG_\ell$ of
the adjoint representation of $\GL_{n,\F_\ell}$.  
We want to apply Proposition~\ref{primetonon} 
with $\Gamma:=\Gamma_\ell$ and $\Delta := \Gamma'_\ell$.
The elements $\gamma\in \Gamma_\ell$ which reduce (mod $\ell$) to $\bar\gamma$ are
Zariski-dense in $\bG_\ell(\Q_\ell)$, so we may choose $\gamma$ to satisfy properties (a) and (b).
By Theorem~\ref{thmA} (i), the formal characters of $\uG_\ell\to \GL_{n,\F_\ell}$ and $\bG_\ell\to \GL_{n,\Q_\ell}$ are
the same, so in particular, two eigenvalues $\lambda_1$ and $\lambda_2$ of $\gamma$ are equal if $\lambda_1-\lambda_2$ is not
a unit.
We can therefore apply the proposition, and ($\ast$) follows.\end{proof}

Without ($\ast$), we can prove much less.  Nevertheless, we still have the following.

\begin{prop}\label{FT}
Let $\{\rho_\ell\}_\ell$ be the system of $\ell$-adic representations
arising from the $i$th $\ell$-adic cohomology of a proper smooth variety $X$
defined over a number field $K$ which is sufficiently large. Then for sufficiently large $\ell$,
\begin{enumerate}[(i)]
\item the reductive group $\bG_\ell^{\red}$ splits over some finite unramified extension of $\Q_\ell$;
\item the reductive group $\bG_\ell^{\red}$ is unramified over every degree $12$ totally ramified extension of $\Q_\ell$.
\end{enumerate}
\end{prop}

Recall that, according to standard terminology, a connected reductive group over a local field which splits over an unramified extension of that field need not be unramified, since it need not have a rational Borel subgroup.

\begin{proof}
The first assertion follows immediately by the method of Frobenius tori (see \cite{Se81,Ch92,LP97}).
The second assertion follows from the first and the last statement in Theorem~\ref{gpmain}.
\end{proof}

\subsection{Proof of Theorem \ref{main2}(a)}
\begin{proof}
Let $X$ be an abelian variety defined over a subfield $K$ of $\C$
that is finitely generated over $\Q$.
Since the $\ell$-adic representation $\rho_\ell$ arising from $H^i(X_{\overline K},\Q_\ell)$
is semisimple by Faltings \cite{Fa83}, the algebraic monodromy group $\bG_\ell$ is reductive.

We first treat the case $i=1$.
By taking a finite extension of $K$, we may assume $\bG_\ell$ is connected for all $\ell$ (see e.g., \cite[$\mathsection2.3$]{Ca15}).
There exists an abelian scheme $f:\mathcal{X}\to\mathcal{S}$ defined over some number field 
whose generic fiber is $X\to\mathrm{Spec}K$.
Let $s$ be a closed point of $\mathcal{S}$ and  $R^1f_*\Q_\ell$ the lisse sheaf on $\mathcal{S}$.
Then by the proper-smooth base change theorem, $\rho_\ell$ factors through the $\ell$-adic representation 
$$\psi_\ell:\pi_1^{\et}(\mathcal{S},\bar s)\to \GL(R^1f_*\Q_\ell|_{\bar s})$$
for all $\ell$. 
Hence, we may assume $\Gamma_\ell$ (resp. $\bG_\ell$) is the monodromy group (resp. algebraic monodromy group) of $\psi_\ell$. 
Moreover, the composition $\psi_{\ell,s}:=\psi_\ell\circ (\pi_1^{\et}(s,\bar s)\to \pi_1^{\et}(\mathcal{S},\bar s))$
is isomorphic to the $\ell$-adic representation $H^1(\mathcal{X}_{\bar s},\Q_\ell)$ 
of the abelian variety $\mathcal{X}_s$ (the fiber over $s$)
defined over the residue field of $s$ (some number field) for all $\ell$.
Let $\Gamma_{\ell,s}$ (resp. $\bG_{\ell,s}$) be the monodromy group (resp. algebraic monodromy group) of $\psi_{\ell,s}$.
We may identify $\Gamma_{\ell,s}$ (resp. $\bG_{\ell,s}$) as a subgroup of $\Gamma_\ell$ (resp. $\bG_{\ell}$).

Fix a prime $p$, one can find a closed point $s$ of $\mathcal{S}$ such that $\bG_{p,s}=\bG_p$ \cite{Se81}.
By the main theorem of \cite{Hu12}, we have $\bG_{\ell,s}=\bG_\ell$ for all $\ell$.
Therefore, it suffices to deal with the case when $K$ is a number field.
Since the condition $(\ast)$ holds
by the Tate conjecture for abelian varieties proved by Faltings in the strong form given in \cite[Theorem~4.2]{FW84},
we are done by Theorem \ref{main}.

Since we have $H^i(X_{\overline K},\Q_\ell)\cong\bigwedge^i H^1(X_{\overline K},\Q_\ell)$
as $\Gal_K$-representations, the general case follows from the lemma below.
\end{proof}

\begin{lemma}\label{hypermorph}
Let $\mu:\bG\to \bH$ be a surjective morphism between connected reductive algebraic groups defined over $\Q_\ell$
and $\Gamma$ a compact subgroup of $\bG(\Q_\ell)$.
If $\Gamma^{\sc}$ is a hyperspecial maximal compact subgroup of $\bG^{\sc}(\Q_\ell)$,
then $\mu(\Gamma)^{\sc}$ is a hyperspecial maximal compact subgroup of $\bH^{\sc}(\Q_\ell)$.
\end{lemma}

\begin{proof}
The surjective $\Q_\ell$-morphism $\mu:\bG\to\bH$ induces a surjective $\Q_\ell$-morphism 
$\mu^{\sc}:\bG^{\sc}\to\bH^{\sc}$ mapping $\Gamma^{\sc}$ into $\mu(\Gamma)^{\sc}$. 
Since both $\bG^{\sc}$ and $\bH^{\sc}$ are simply connected,
$\bH^{\sc}$ can be identified as a direct factor of $\bG^{\sc}$ and $\mu^{\sc}$ can be identified as the projection to the factor.
It follows that $\mu^{\sc}(\Gamma^{\sc})$ is also a hyperspecial maximal compact subgroup of $\bH^{\sc}(\Q_\ell)$. 
Since  $\mu^{\sc}(\Gamma^{\sc})\subset \mu(\Gamma)^{\sc}$ holds, 
the compact subgroup $\mu(\Gamma)^{\sc}$ is equal to $\mu^{\sc}(\Gamma^{\sc})$ and we are done.
\end{proof}

\subsection{Proof of Theorem \ref{main2}(b)}
Let $X$ be a hyperk\"ahler variety defined over a subfield $K$ of $\C$
that is finitely generated over $\Q$ and $\rho_\ell$ the $\ell$-adic representation 
arising from $H^2(X_{\overline K},\Q_\ell)$. If the dimension $n$ of the representation is less than
$4$, then $\bG^{\sc}$ is trivial or of type A for all $\ell$.
Hence, Theorem \ref{main2}(b) follows by \cite[Theorem 15]{HL16}.

Suppose $n\geq 4$ and write $X_\C:=X\times_K\C$.
By the Kuga-Satake construction, there is a complex abelian variety $A_\C$ and a surjective morphism  
\begin{equation}\label{Hodge}
H^1(A_\C,\Q)\otimes H^1(A_\C,\Q)\to H^2(X_\C,\Q)
\end{equation}
of pure Hodge structures, i.e.,
there is a Hodge cycle of 
$$H^1(A_\C,\Q)^*\otimes H^1(A_\C,\Q)^*\otimes H^2(X_\C,\Q)$$ 
giving the correspondence \eqref{Hodge}.
Assume $A_\C$ has a model $A_L$ defined over a subfield $L$ of $\C$ that is finitely generated over $K$.
Since such Hodge cycles are motivated on the product $X_\C\times A_\C$ \cite[Corollary 1.5.3]{An96a}
and motivated cycles are absolute Hodge \cite[Proposition 2.5.1]{An96b}
in the sense of Deligne \cite[2.10]{De82},
we obtain for each $\ell$ a surjective morphism  
$$H^1(A_{\overline L},\Q_\ell)\otimes H^1(A_{\overline L},\Q_\ell)\to H^2(X_{\overline L},\Q_\ell)$$
of $\Gal_L$-representations (after replacing $L$ by a finite extension if necessary \cite[Proposition 2.9(b)]{De82}).
Since $A_L$ is an abelian variety defined over a subfield $L$ of $\C$
that is finitely generated over $\Q$,
the assertion of Theorem \ref{main2} holds for $X_L:=X\times_K L$
by Theorem \ref{main2}(a) and Lemma \ref{hypermorph}.
By induction, it suffices to show that Theorem \ref{main2} for $X/K$ also holds 
when $L$ is a finite extension of $K$ or $L=K(t)$ where $t$ is transcendental over $K$.
The former case is obvious and the latter case can be done by the fact that $\pi_1^{\et}(\mathrm{Spec}\overline K[t])$
is trivial.\qed

\vspace{.1in}

\begin{thebibliography}{99}

\bibitem[An96a]{An96a}
	Andr\'e, Yves:
	On the Shafarevich and Tate conjectures for hyperk\"ahler varieties,
	\textit{Math. Ann.} \textbf{305} (1996), 205--248.


\bibitem[An96b]{An96b}
	Andr\'e, Yves:
	Pour une th\'eorie inconditionnelle des motifs.
	\textit{Inst.\ Hautes\ \'Etudes Sci.\ Publ.\ Math.} \textbf{83} (1996), 5--49.

	
\bibitem[BT87]{BT}
	Bruhat, F.; Tits, J.:
	Groupes alg\'ebriques sur un corps local. Chapitre III. Compl\'ements et applications \`a la cohomologie galoisienne. 
	\textit{J. Fac. Sci. Univ. Tokyo Sect. IA Math.} \textbf{34} (1987), no.\ 3, 671--698.
	
	
\bibitem[BdS49]{BdS49}
	Borel, Armand; de Siebenthal, Jean:
	Les sous-groupes ferm\'es de rang maximum des groupes de Lie clos, 
	\textit{Commentarii mathematici Helvetici} \textbf{23} (1949): 200--221.
				

\bibitem[Ca15]{Ca15}
  Cadoret, Anna:
	An open adelic image theorem for abelian schemes, \textit{IMRN} Vol.\ \textbf{2015}, 10208--10242.

\bibitem[CM19]{CM19}
	Cadoret, Anna; Moonen, Ben:
	Integral and adelic aspects of the Mumford-Tate conjecture,  
	to appear in \textit{J.\ Inst.\ Math.\ Jussieu}, arXiv: 1508.0642.

\bibitem[CHT17]{CHT17}
Cadoret, Anna; Hui, Chun Yin; Tamagawa, Akio:
Geometric monodromy -- semisimplicity and maximality, \textit{Annals of Math.} \textbf{186}, p. 205--236, 2017.

\bibitem[CHT18]{CHT18}
Cadoret, Anna; Hui, Chun Yin; Tamagawa, Akio:
$\Q_\ell$- versus $\F_\ell$-coefficients in the Grothendieck-Serre-Tate conjectures, preprint.

\bibitem[Ch92]{Ch92} 
Chi, Wen-Chen: 
$\ell$-adic and $\lambda$-adic representations associated to abelian varieties defined
over number fields, \textit{Amer. J. Math.} \textbf{114} (1992) 315--353.
  	
\bibitem[Co14]{Co14}
Conrad, Brian:
\textit{Reductive group schemes}, online notes, 2014.

\bibitem[CR88]{CR88} 
  Curtis, Charles W.; Reiner, Irving: 
  Representation theory of finite groups and associative algebras, reprint of the 1962 original. 
  Wiley Classics Library. A Wiley-Interscience Publication. \textit{John Wiley $\&$ Sons, Inc., New York}, 1988.

\bibitem[De74]{De74} 
  Deligne, Pierre: 
  La conjecture de Weil I, 
  \textit{Publ. Math. I.H.E.S.}, \textbf{43} (1974), 273-307.
  
\bibitem[De82]{De82}
  Deligne, Pierre:
	Hodge cycles on abelian varieties (notes by J. Milne), in
	\textit{Hodge cycles, motives, and Shimura varieties}, volume \textbf{900} of Lecture Notes in Mathematics, Springer-Verlag, Berlin, 1982.
	


\bibitem[Fa83]{Fa83}
  Faltings, Gerd:
  Endlichkeitss$\ddot{\mathrm{a}}$tze f$\ddot{\mathrm{u}}$r abelsche Variet$\ddot{\mathrm{a}}$ten $\ddot{\mathrm{u}}$ber Zahlk$\ddot{\mathrm{o}}$rpern. 
  \textit{Inventiones Mathematicae} \textbf{73} (1983), 
  no.\ 3, 349--366.

\bibitem[FW84]{FW84} 
  Faltings, Gerd; W$\ddot{\mathrm{u}}$stholz, Gisbert (eds.):
  Rational Points, Seminar Bonn/Wuppertal 1983-1984, Vieweg 1984.
	 	 
\bibitem[Gi10]{Gi}
Gilles, Philippe: The Borel--de Siebenthal's Theorem,
\texttt{http://math.univ-lyon1.fr/homes-www/gille/prenotes/bds.pdf}.

\bibitem[Gr66]{EGAIV}
Grothendieck, Alexander: 
\'El\'ements de g\'eom\'etrie alg\'ebrique (r\'edig\'e avec la collaboration de Jean Dieudonn\'e): IV. \'Etude locale des sch\'emas et des morphismes de sch\'emas,
Troisi\`eme partie, \textit{Publ.~Math. IHES} \textbf{28} (1966), 5--255.


	
\bibitem[Hu12]{Hu12} 
  Hui, Chun Yin: 
  Specialization of monodromy group and $\ell$-independence, 
  \textit{C. R. Acad. Sci. Paris}, Ser. I, \textbf{350} (2012), 5--7.

\bibitem[Hu13]{Hu13}
	Hui, Chun Yin:
	Monodromy of Galois representations and equal-rank subalgebra equivalence,
	\textit{Math. Res. Lett.} \textbf{20} (2013), no. 4, 705--728.

\bibitem[Hu15]{Hu15}
	Hui, Chun Yin:
	$\ell$-independence for compatible systems of (mod $\ell$) Galois representations,
	\textit{Compositio Mathematica} \textbf{151} (2015), 1215--1241. 
	
\bibitem[Hu18]{Hu18}
  Hui, Chun Yin:
	On the rationality of certain type A Galois representations,
	\textit{Transactions of the American Mathematical Society}, Volume \textbf{370}, Number 9, September 2018, p. 6771--6794.
	
	
\bibitem[HL15]{HL15} 
  Hui, Chun Yin; Larsen, Michael: 
  Adelic openness without the Mumford-Tate conjecture,
 arXiv:1312.3812.

\bibitem[HL16]{HL16} 
  Hui, Chun Yin; Larsen, Michael: 
  Type A images of Galois representations and maximality,
  \textit{Mathematische Zeitschrift} \textbf{284} (2016), no.3--4, 989--1003.
	
\bibitem[Ja97]{Ja}
	Jantzen, Jens Carsten:
	Low-dimensional representations of reductive groups are semisimple. 
	Algebraic groups and Lie groups, 255--266, Austral. Math. Soc. Lect. Ser., 9, 
	Cambridge Univ. Press, Cambridge, 1997.
	
\bibitem[Jo78]{Jo78} 
  Jordan, Camille: 
  M\'emoire sur les \'equations differentielles lin\'eaires \`a int\'egrale alg\`ebrique, 
  \textit{J. f$\ddot{u}$r Math.} \textbf{84} (1878), 89--215.
		
\bibitem[La95a]{La-SS}
	Larsen, Michael:
	On the semisimplicity of low-dimensional representations of semisimple groups in characteristic 	$p$. \textit{J.\ Algebra} \textbf{173} (1995), no.\ 2, 219--236.

\bibitem[La95b]{La95} 
	Larsen, Michael:
	Maximality of Galois actions for compatible systems. 
	\textit{Duke Math.\ J.} 
	\textbf{80} (1995), no.\ 3, 601--630. 
	
\bibitem[La10]{La10}
	Larsen, Michael:
	Exponential generation and largeness for compact $p$-adic Lie groups. 
	\textit{Algebra Number Theory} \textbf{4} (2010), no.\ 8, 1029--1038.

	

\bibitem[LP97]{LP97} 
	Larsen, Michael; Pink, Richard: 
 	A connectedness criterion for $\ell$-adic Galois representations,
	\textit{Israel J. Math.} \textbf{97} (1997), 1--10.


	
\bibitem[LP11]{LP}
	Larsen, Michael; Pink, Richard: 
	Finite subgroups of algebraic groups. \textit{J.\ Amer.\ Math.\ Soc.} \textbf{24} (2011), no.\ 4, 1105--1158.

	
	
\bibitem[No87]{No87}
	 Nori, Madhav V.:
	 On subgroups of $\GL_n(\F_p)$. 
	 \textit{Invent.\ Math.}
	 \textbf{88} (1987), no.\ 2, 257--275.
	 	
	
\bibitem[Ri83]{Ri}
	Ribet, Kenneth A.:
	Hodge classes on certain types of abelian varieties. 
	\textit{Amer.\ J.\ Math.} \textbf{105} (1983), no.\ 2, 523--538.
	
	

\bibitem[Se65]{CG}
	Serre, Jean-Pierre:
	Cohomologie galoisienne. With a contribution by Jean-Louis Verdier. Lecture Notes in Mathematics, No. 5. Troisi\`eme \'edition, 1965 Springer-Verlag, Berlin-New 	York, 1965.
	 
	 
\bibitem[Se72]{Se72}
	Serre, Jean-Pierre:
	Propri\'et\'es galoisiennes des points d'ordre fini des courbes elliptiques. 
	\textit{Invent.\ Math.}
	\textbf{15} (1972), no.\ 4, 259--331.

\bibitem[Se79]{SeLF}
	Serre, Jean-Pierre:
	Local fields.  
	Graduate Texts in Mathematics, 67. Springer-Verlag, New York-Berlin, 1979.
	
\bibitem[Se81]{Se81}
	Serre, Jean-Pierre:
	Lettre \`a Ken Ribet du 1/1/1981 et du 29/1/1981 (Oeuvres IV, no.\ 133).

\bibitem[Se84]{Se84a}
	Serre, Jean-Pierre:
	Lettre \`a Daniel Bertrand du 8/6/1984 (Oeuvres IV, no.\ 134).
	
\bibitem[Se85]{Se85}
	Serre, Jean-Pierre:
	R\'esum\'e des cours de 1984--1985 (Oeuvres IV, no.\ 135).


	
\bibitem[Se86]{Se86a} 
  Serre, Jean-Pierre:
  Lettre \'a Marie-France Vign\'eras du 10/2/1986 (Oeuvres IV, no.\ 137).
	
\bibitem[Se94]{Se94}
	Serre, Jean-Pierre:
	Propri\'et\'es conjecturales des groupes de Galois motiviques et des repr\'esentations $l$-adiques. Motives (Seattle, WA, 1991), 377--400, Proc. Sympos. Pure Math., 55, Part 1, Amer. Math. Soc., Providence, RI, 1994. 

\bibitem[Se98]{Se98} 
	Serre, Jean-Pierre:
	Abelian $l$-adic representation and elliptic curves, Research Notes in Mathematics Vol. 7 (2nd ed.), \textit{A K Peters} (1998).


\bibitem[Sp68]{Springer}
	Springer, T. A.:
	Weyl's character formula for algebraic groups. 
	\textit{Invent.\ Math.} \textbf{5} (1968), 85--105.
	

\bibitem[Ti66]{Ti66}
Tits, Jacques:
	Classification of algebraic semisimple groups. Proc. Summer Inst. on Algebraic Groups and Discontinuous Groups
	(Boulder, 1965), Proc. Sympos. Pure Math., vol. 9, Amer. Math. Soc., Providence, R.I., 1966, pp. 33--62.
	
\bibitem[Ti79]{Ti79}
Tits, Jacques:
	Reductive groups over local fields. Automorphic forms, representations and L-functions, Part 1, 	
	pp.\ 29--69, Proc. Sympos. Pure Math., XXXIII, Amer. Math. Soc., Providence, R.I., 1979.	
	
\bibitem[Va03]{Va03}
	Vasiu, Adrian: 
	Surjectivity criteria for p-adic representations I, 
	\textit{Manuscripta Math.} \textbf{112} (2003), no.\ 3, 325--355. 
	
\bibitem[Wi02]{Wi}
	Winterberger, J.-P.:
	D\'emonstration d'une conjecture de Lang dans des cas particuliers,
	\textit{J.\ reine angew.\ Math.} \textbf{553} (2002), 1--16.
	
\end{thebibliography}
\end{document}